\documentclass[12pt,leqno]{amsart}

\usepackage{amsmath,amssymb,amsthm}
\usepackage{amsfonts}
\usepackage{verbatim}
\usepackage{graphicx}
\usepackage[utf8]{inputenc}
\usepackage[english]{babel}
\usepackage{hyperref}

\newtheorem{theorem}{Theorem}
\newtheorem*{theorem*}{Theorem}
\newtheorem{lemma}{Lemma}
\newtheorem{corollary}{Corollary}
\newtheorem{proposition}{Proposition}

\newtheorem*{example*}{Example}

\theoremstyle{definition}
\newtheorem*{definition*}{Definition}
\newtheorem{definition}{Definition}

\newtheorem*{conjecture}{Conjecture}
\theoremstyle{remark}
\newtheorem{remark}{Remark}

\newcommand{\bbR}{\mathbb{R}}

\newcommand{\bbN}{\mathbb{N}}
\newcommand{\bbS}{\mathbb{S}}
\newcommand{\mF}{\mathcal{F}}

\newcommand{\mT}{\mathcal{T}}

\newcommand{\mcP}{\mathcal{P}}

\newcommand{\Rn}{\bbR^d}
\newcommand{\Sn}{\bbS^{d-1}}

\newcommand{\ddv}[1]{\frac{\partial}{\partial #1}}

\newcommand{\const}{\mathop{{\rm const}}\nolimits}
\newcommand{\Lip}{\mathop{{\rm Lip}}\nolimits}
\newcommand{\1}{1\hspace{-.55ex}\mbox{l}}
\newcommand{\Leb}{\mathrm{Leb}}

\author{Victor Kleptsyn, Aline Kurtzmann}
\title{Ergodicity of self-attracting motion}

\begin{document}

\maketitle

\thanks{The second author has been partially supported by the Swiss National Science Foundation grant PBNE2-119027.}

\begin{abstract}
We study the asymptotic behaviour of a class of self-attracting motions on $\bbR^d$. We prove the decrease of the free energy related to the system and mix it together with stochastic approximation methods. We finally obtain the (limit-quotient) ergodicity of the self-attracting diffusion with a speed of convergence. 
\end{abstract}

\tableofcontents

\section{Introduction}\label{s:introduction}

\subsection{Statement of the problem}

This text is devoted to study the asymptotic behaviour of a
Brownian motion, interacting with its own passed trajectory,
so-called ``self-interacting motion''. Namely, we fix an
interaction potential function $W:\Rn\to \bbR$, and consider the
stochastic differential equation
\begin{equation}\label{eq:AABM1}
\mathrm{d}X_t= \sqrt{2}\, \mathrm{d}B_t - \left(\frac{1}{t} \int_0^t \nabla W(X_t-X_s)\, \mathrm{d}s \right) \mathrm{d}t,
\end{equation}
where $(B_t,t\ge 0)$ is a standard Brownian motion, with an initial condition
of given~$X_0$ (with the condition of continuity at $t=0$). This
equation can be rewritten using the normalized occupation measure
$\mu_t$:
$$
\mu_t= \frac{1}{t} \int_0^t \delta_{X_s} \, \mathrm{d}s,
$$
where $\delta_x$ is the Dirac measure concentrated at the
point~$x$. Using this convention, the equation~\eqref{eq:AABM1}
becomes
\begin{equation}\label{eq:AABM2}
\mathrm{d}X_t= \sqrt{2}\, \mathrm{d}B_t - \nabla W*\mu_t (X_t) \, \mathrm{d}t,
\end{equation}
where $*\,$ stands for the convolution.

Note that the equations~\eqref{eq:AABM1}, \eqref{eq:AABM2}
clearly have singularities at $t=0$, which is the reason why
sometimes they are considered only after some positive
time $r>0$. We discuss the existence and uniqueness questions
for the solution in the appendix. \medskip

Similar problems have already been studied since the 90's, for instance by Durrett and Rogers \cite{DR}, or Bena\"im, Ledoux and Raimond \cite{BLR}, initially to modelize the evolution of polymers or ants. The first
time-continuous self-interacting processes have been introduced by
Durrett and Rogers \cite{DR} under the name of ``Brownian polymers".
They are solutions to SDEs of the form
\begin{equation}\label{eq:DR}
\mathrm{d}X_t = \mathrm{d}B_t + \left(\int_0^t f(X_t-X_s)\,
\mathrm{d}s\right) \mathrm{d}t
\end{equation}
where $(B_t, t\geq 0)$ is a
standard Brownian motion and $f$ a given function. We remark that, in the latter equation, the drift term is given by the non-normalized measure $t\mu_t$ and not by $\mu_t$ as the process we will study here. As the process $(X_t, t\geq 0)$ evolves in an environment changing with its past trajectory, this SDE defines a
self-interacting diffusion, which can be either self-repelling or
self-attracting, depending on the function $f$. In any dimension, Durrett \& Rogers obtained that the upper limit of $|X_t|/t$ does not exceed a deterministic constant whenever $f$ has a compact
support. Nevertheless, very few results are known as soon as the interaction is not self-attracting.

Self-interacting diffusions, with dependence on the (convoled) empirical measure $(\mu_t, t\geq 0)$, have been considered since the work of Bena\"im, Ledoux \& Raimond \cite{BLR}. A great difference between these diffusions and Brownian polymers is that the drift term is divided by $t$. This implies that the long-time away interaction is less important than the near-time interaction (the interaction is not ``uniform in time'' anymore). Bena\"im et al.  have shown in \cite{BLR,BR} that the asymptotic behaviour of $\mu_t$ can be related to
the analysis of some deterministic dynamical flow defined on the
space of the Borel probability measures. Afterwards, one can go
further in this study and give sufficient conditions for the a.s.
convergence of the empirical measure. It happens that, with a symmetric interaction, $\mu_t$ converges a.s. to a local minimum of a nonlinear free energy functional (each local minimum having a positive probability to be chosen), this free energy being a Lyapunov function for the deterministic flow. %The authors did not obtain any rate of convergence. 
These results are valid for a compact manifold. Part of them have recently been generalized to $\mathbb{R}^d$ (see \cite{AK}) assuming a confinement potential satisfying some conditions --- these hypotheses on the confinement potential are required since in general the process can be transient, and is thus very difficult to analyze. In these works, no rate of convergence is obtained. Most of these results are summarized in a recent survey of
Pemantle \cite{Pem}, which also includes self-interacting random walks.

Coming back to the process introduced by Durrett \& Rogers, all the results obtained have in common that the drift may overcome the noise, so that the randomness of the process is ``controlled". To illustrate that, let us mention, for the same model of Durrett \& Rogers, the case of a repulsive and compactly supported function $f$, that was conjectured in \cite{DR} and has been partially solved very recently by Tarr\`es, T\'oth and Valk\'o \cite{TTV}:
\begin{conjecture}[Durrett \& Rogers \cite{DR}]
Suppose that $f: \mathbb{R}\rightarrow \mathbb{R}$ is an odd function of compact support, such that $xf(x)\ge 0$. Then, for the process $X$ defined by~\eqref{eq:DR}, the quotient $X_t /t$ converges a.s. to 0.
\end{conjecture}
In~\eqref{eq:AABM1}, the drift term is divided by $t$, and so it is bounded for a compactly supported interaction $W$. As for the process of the conjecture, the interaction potential is in general not strong enough for the process~\eqref{eq:AABM1} to be recurrent, and the behaviour is then very difficult to analyze. In particular, it is hard to predict the relative importance of the drift term (in competition with the Brownian motion) in the evolution. 

On the other hand, in our case of uniformly convex $W$, the interaction potential is attractive enough for the diffusion (a bit modified) to be comparable to an Ornstein-Uhlenbeck process, what gives an access to its ergodic behaviour.\medskip

Another problem, related to the one considered in this paper, is the diffusion corresponding to MacKean and Vlasov's PDE. Namely, consider the Markov process defined by the SDE
\begin{equation}\label{eq:MKV}
\mathrm{d}Y_t= \sqrt{2}\, \mathrm{d}B_t - \nabla W*\nu_t(Y_t)\, \mathrm{d}t,
\end{equation}
where $\nu_t$ stands for the \emph{law} of $Y_t$, and $W$ is a smooth strictly uniformly convex function. 

The questions of the asymptotic law for $Y$ have been intensively studied these last years, by Carrillo, MacCann \& Villani~\cite{CMV}, Bolley, Guillin \& Villani~\cite{BGV}, or Cattiaux, Guillin \& Malrieu~\cite{CGM} for instance. It turns out that, under some assumptions, the laws $\nu_t$ converge to the limit measure $\nu^*$. This measure is characterized as a fixed point of a map $\Pi: \nu \mapsto \Pi(\nu)$ associating to a measure $\nu$ the probability measure 
$$\Pi(\nu)(\mathrm{d}x) := \frac{1}{Z} e^{-W*\nu(x)} \mathrm{d}x,$$ which is the stationary measure of the process, with $\nu_t$ in the right-hand side of~\eqref{eq:MKV} replaced by $\nu$ and $Z=Z_\nu$ is the normalization constant.

 In particular, Carrillo, MacCann \& Villani~\cite{CMV} have shown, using some mass transport tools, that
the relative free energy corresponding to~$\nu_t$ with respect to~$\nu^*$ decreases exponentially fast to 0. Then Talagrand's inequality allows to compare the relative free energy to the Wasserstein distance in case of uniform convexity of the interaction potential $W$, and so they have obtained the decrease to 0 of the quadratic Wasserstein distance between~$\nu_t$ and~$\nu^*$. %Note that Bolley, Guillin \& Villani~\cite{BGV}, or Cattiaux, Guillin \& Malrieu~\cite{CGM} have also studied the asymptotics of the Markov process~\eqref{eq:MKV} by using some probabilistic methods.

We remark that a huge difference between the preceding Markov
process and the (non-Markov) self-interacting diffusion is that the
asymptotic $\sigma$-algebra is in general not trivial for the
non-Markov process. Nevertheless, we will use a similar mass transport method to show the convergence of the empirical measure $\mu_t$.

\subsection{Main results}\label{ss:Main}

Our results are analogous to those of Carrillo et al. \cite{CMV}: under some
assumptions imposed on the interaction potential~$W$, we show that
the empirical measure~$\mu_t$ almost surely converges to
an equilibrium state, which is unique up to translation:
\begin{theorem}[Main result]\label{t:main}
Suppose, that $W\in C^2(\Rn)$, and:
\begin{itemize}
    \item[1)]\label{i:thma1} spherical symmetric: $W(x)=W(|x|)$;
    \item[2)]\label{i:thma2} uniformly convex: denoting by $\Sn$ the $(d-1)-$dimensional sphere, 
    \begin{equation}
    \exists C_W>0 : \quad \forall x\in \Rn, \forall v\in \Sn ,
    \qquad \left.\frac{\partial^2 W}{\partial v^2}\right|_{x}\ge C_W;
    \end{equation}
    \item[3)]\label{i:thma3} $W$ has at most a polynomial growth: for some polynomial $P$, we have 
    \begin{equation}\label{eq:domination}
    \forall x\in\Rn \quad \vert W(x)\vert + |\nabla W(x)| + \|\nabla^2 W(x)\|\leq P(|x|);
    \end{equation}
\end{itemize}
Then, there exists a unique symmetric 
density $\rho_{\infty}:\Rn\to\bbR_+$, such that almost surely, there exists $c_\infty$ such that  
$$
\mu_t := \frac{1}{t} \int_0^t \delta_{X_s}\mathrm{d}s \xrightarrow[t\to\infty]{*-weakly}
\rho_{\infty}(x-c_{\infty}) \, \mathrm{d}x.
$$
Moreover, there exists $a>0$ such that the speed of convergence of $\mu_t$ toward $\rho_\infty(\cdot +c_\infty)$ for the Wasserstein distance is at least $\exp\{-a\sqrt[k+1]{\log t}\}$, where $k$ is the degree of $P$.
\end{theorem}

\begin{remark}\label{rk:spherical}
The assumption~\ref{i:thma1}) corresponds to the physical assumption of the interaction force between two particles being directed along the line joining them, and to the third Newton's law (that is the equality between the action and the reaction forces). The symmetry assumption cannot be omitted, as shows an example in the appendix.%~\ref{ss:non-sym}. 
\end{remark}
\begin{remark}
We will suppose in the following, without any loss of generality, that $P\ge 1$ is of degree $k\ge 2$ and such that for all $x,y\in \Rn$, we have $P(\vert x-y\vert) \le P(\vert x\vert)P(\vert y\vert)$. Indeed, we choose $P(\vert x|) = A (1 +|x|^k)$, where $A$ is a constant large enough. This will be used in~\S\ref{ss:measures}.
\end{remark}
The origin of the following remark will be clear after the
discussion in~\S\ref{s:free-energy}%\ref{ss:entropy}:
\begin{remark}
The density $\rho_{\infty}$ is the same limit density as in the
result of \cite{CMV}, uniquely defined (among the centered
densities) by the following property: $\rho_{\infty}$ is a
positive function, proportional to~$e^{-W*\rho_{\infty}}$.
\end{remark}

We can also consider the same drifted motion in presence of an external potential $V$. For this, the following result is a generalization of Theorem~\ref{t:main} (where we replace $C_W$ by $C$ in the notation): 
\begin{theorem}\label{t:main-2}
Let $X$ be the solution to the equation
\begin{equation}\label{eq:eds-V}
\mathrm{d}X_t = \sqrt{2} \mathrm{d}B_t -\left( \nabla V(X_t) + \frac{1}{t}\int_0^t \nabla W(X_t-X_s) \mathrm{d}s\right) \,\mathrm{d}t.
\end{equation}
Suppose, that $V\in C^2(\Rn)$ and $W\in C^2(\Rn)$, and:
\begin{itemize}
    \item[1)] spherical symmetric: $W(x)=W(|x|)$;
    \item[2)] $V$ and $W$ are convex, $\lim_{|x|\rightarrow \infty}V(x) = +\infty$, and either $V$ or $W$ is uniformly convex:
$$
    \exists C>0 : \quad  \forall x\in \Rn, \forall v\in \Sn , \, \left.\frac{\partial^2 V}{\partial v^2}\right|_{x}\ge C\quad \text{ or } \quad  \forall x, \forall v,\ \left.\frac{\partial^2 W}{\partial v^2}\right|_{x}\ge C;
$$
    \item[3)] $V$ and $W$ have at most a polynomial growth: for some polynomial $P$ we have $\forall x\in\Rn$
    \begin{equation}\label{eq:domination-2}
    \vert V(x)\vert + \vert W(x)\vert + |\nabla V(x)| + |\nabla W(x)| + \|\nabla^2 V(x)\| +  \|\nabla^2 W(x)\|\leq P(|x|).
    \end{equation}
\end{itemize}
Then, there exists a unique density $\rho_{\infty}:\Rn\to\bbR_+$, such that almost surely
$$
\mu_t = \frac{1}{t}\int_0^t \delta_{X_s}\mathrm{d}s \xrightarrow[t\to\infty]{*-weakly}
\rho_{\infty}(x) \, \mathrm{d}x.
$$
\end{theorem}
As the proof of the latter Theorem coincides with the proof of Theorem~\ref{t:main} almost identically, we do not present it here. It suffices to add $V$ in the arguments below. Moreover, if $V$ is symmetric with respect to some point $q$, then the corresponding density $\rho_\infty$ is also symmetric with respect to the same point $q$.

The proof of Theorem \ref{t:main} is split into two parts. Consider a
natural ``reference point'' for a measure $\mu$:
\begin{definition}\label{def:mu-c}
Consider a measure $\mu$ on $\Rn$, decreasing fast enough for $W*\mu$ to be defined. The \emph{center} of $\mu$ is the point
$c_\mu=c(\mu)$ such that $\nabla W*\mu (c_\mu)=0$, or equivalently, the point
where the convolution $W*\mu$ (the potential generated by $\mu$)
takes its minimal value. Also, we define the \emph{centered} measure
$\mu^c$ as the translation of the measure $\mu$, bringing $c_\mu$
to the origin:
\begin{equation}
\mu^c(A)=\mu(A+c_\mu).
\end{equation}
\end{definition}
\begin{remark}
This notion of center had been previously introduced by Raimond in \cite{OTwist}. Indeed, to study the linear attracting $d$-dimensional case of Brownian polymers, Raimond has defined the center and proved that the process remains close to $c_t=c(\mu_t)$ (and that $c_t$ converges a.s.).\\ 
A sufficient condition for the existence of the center is that $W$ is convex, and it is unique if $W$ is stricty convex.
\end{remark}

The first part of the proof of Theorem~\ref{t:main} consists in 
proving the convergence of centered occupation measures:
\begin{theorem}\label{t:centered}
Under the assumptions of Theorem~\ref{t:main}, for some symmetric density
function $\rho_{\infty}:\Rn\to\bbR_+$, we have almost surely
$$
\mu_t^c \xrightarrow[t\to\infty]{*-weakly} \rho_{\infty}(x) \, \mathrm{d}x
$$
\end{theorem}
The second is the convergence of centers:
\begin{theorem}\label{t:centers}
Under the assumptions of Theorem~\ref{t:main}, almost surely the
centers $c_t:=c(\mu_t)$ converge to some (random) limit $c_{\infty}$.
\end{theorem}
It is clear that the two latter theorems imply the main result. Let us sketch their proofs.

%\newpage

\subsection{Outline of the proof and physical interpretation}\label{ss:outline}

\subsubsection{Existence and uniqueness}
First, a standard remark is Markovianization: the behaviour of the
pair $(X_t,\mu_t)$ is Markovian. The reader will find it, together with some
other standard remarks, in \S\ref{ss:Markov}. Unfortunately, the
Markov process $(X_t,\mu_t)$ is infinite-dimensional and, in
general (except for the case of a polynomial interaction $W$), we do not manage to reduce to a finite-dimensional process. So, we
do not use this information directly in order to obtain
interesting properties on $\mu_t$, because the state space is then
too large.

After this remark, we discuss the global existence and uniqueness for the
solutions of~\eqref{eq:AABM2} in \S\ref{ss:global-exist}.

\subsubsection{Discretization}
A next step is discretization: we take a (well-chosen and
deterministic) sequence of times $T_n\to\infty$, with
$T_n\gg T_{n+1}-T_{n}\gg 1$, and consider the behaviour of the
measures~$\mu_{T_n}$. As $T_n\gg T_{n+1}-T_n$, it is natural to
expect (and we will give the corresponding statement) that the
empirical measures $\mu_t$ on the interval $[T_n,T_{n+1}]$ almost
do not change and thus stay close to $\mu_{T_n}$. So, on this
interval we can approximate the solution $X_t$ of~\eqref{eq:AABM2} by the solution of the same equation
with $\mu_t\equiv \mu_{T_n}$:
$$
\mathrm{d}Y_t = \sqrt{2}\, \mathrm{d}B_t - \nabla W*\mu_{T_n}(Y_t) \, \mathrm{d}t,\quad t\in [T_n,T_{n+1}],
$$
in other words, by a Brownian motion in a potential~$W*\mu_{T_n}$ that does not depend on time.

On the other hand, the series of general term $T_{n+1}-T_n$ increases. So, using Birkhoff Ergodic Theorem\footnote{see for
instance \cite{yosida}, chap. XIII}, we see that the (normalized)
distribution $\mu_{[T_n,T_{n+1}]}$ of values of $X_t$ on these
intervals becomes (as $n$ increases) close to the equilibrium measures
$\Pi(\mu_{T_n})$ for a Brownian motion in the potential
$W*\mu_{T_n}$, where (see \S\ref{s:discretization})
$$
\Pi(\mu)(\mathrm{d}x) := \frac{1}{Z(\mu)} e^{-W*\mu(x)} \,
\mathrm{d}x,\quad Z(\mu):=\int_{\Rn} e^{-W*\mu(x)} \,
\mathrm{d}x.$$
But
$$
\mu_{T_{n+1}} = \frac{T_n}{T_{n+1}} \,\mu_{T_n} +
\frac{T_{n+1}-T_n}{T_{n+1}} \,\mu_{[T_n,T_{n+1}]},
$$
so we have
$$
\mu_{T_{n+1}} \approx \frac{T_n}{T_{n+1}} \,\mu_{T_n} +
\frac{T_{n+1}-T_n}{T_{n+1}} \,\Pi(\mu_{T_n}) = \mu_{T_n} +
\frac{T_{n+1}-T_n}{T_{n+1}} (\Pi(\mu_{T_n})-\mu_{T_n}),
$$
and
$$
\frac{\mu_{T_{n+1}}-\mu_{T_n}}{T_{n+1}-T_n} \approx
\frac{1}{T_{n+1}}(\Pi(\mu_{T_n})-\mu_{T_n}).
$$

This could motivate us to approximate the behaviour of the measures $\mu_t$
by trajectories of the flow (on the infinite-dimensional space of
measures)
\begin{equation}\label{eq:9-3/4}
\dot{\mu}=\frac{1}{t}(\Pi(\mu)-\mu),
\end{equation}
or after a logarithmic change of variable $\theta=\log t$,
\begin{equation}\label{eq:flow}
\mu'=\Pi(\mu)-\mu.
\end{equation}

In fact, it is not \emph{a priori} clear that the flow defined by~\eqref{eq:flow} exists, as the space of measures is infinite-dimensional. Though the flow can be shown to be well defined on a subspace of exponentially decreasing measures, we prefer to avoid all these problems by working directly with the discretization model in \S\ref{s:discretization}. Nevertheless, this flow serves very well in motivating the considered functions and lemmas describing their behaviour, as the discretized procedure we have is in fact the Euler method of finding solutions to~\eqref{eq:flow}.

\subsubsection{Physical interpretation: gas re-distribution}\label{ss:physics}

Before proceeding further, let us give a physical interpretation to the flow~\eqref{eq:flow}, predicting its asymptotic behaviour.
Namely, note that a Brownian motion drifted by some potential~$V$,
$$
\mathrm{d}X_t = \sqrt{2} \mathrm{d}B_t -\nabla V (X_t)\mathrm{d}t,
$$
can be thought as movement of gas particles under this potential, and the stationary probability measure, $m=\frac{1}{Z_V} e^{-V} \mathrm{d}x$, is the density with which the gas becomes distributed after some time passes. So, in dimension one, a discrete approximation to the flow~\eqref{eq:flow} can be seen as follows. We take a tube, filled with $W$-interacting gas, separated in a plenty of very small cells (see Fig.~\ref{fig:1}).

\begin{figure}[!h]
    \begin{center}
       \includegraphics{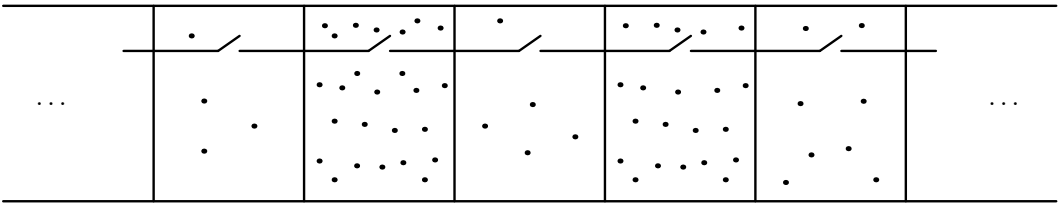}
    \end{center}
\caption{Gas: phase ``separation''}\label{fig:1}
\end{figure}

Each unit of time, small parts (of proportion $\varepsilon$) of gas in these cells 
are separated, allowed to travel along the tube, and are proposed to equilibrate
in the potential generated. This part of all the gas being small, its auto-interaction 
is negligible, thus their new distribution is governed by the field $V:=W*\mu$, generated by the major part of the particles staying fixed to their cells. The small part is then
equilibrated to its weight $\varepsilon$ times $\Pi(\mu)$. 

\begin{figure}[!h]
    \begin{center}
        \includegraphics{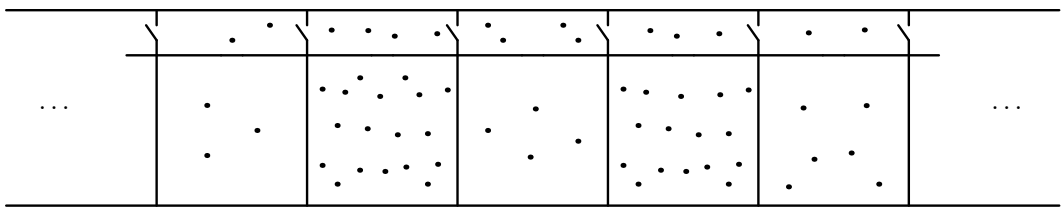}
    \end{center}
\caption{Gas: phase ``re-distribution''}\label{fig:2}
\end{figure}

Then, it is separated again by the cells, thus the distribution after
such step becomes 
$$
(1-\varepsilon) \mu + \varepsilon \Pi(\mu) =
\mu + \varepsilon (\Pi(\mu) -\mu).
$$ 

On the other hand, this procedure does not require any work (in the physical sense) to be
done: the only actions are opening and closing the doors. 
So, due to the general principle, one can expect that the system will tend to its equilibrium. And a tool allowing to show that it is the case is the \emph{free energy}, that we recall in the next paragraph. 

We conclude by noticing that the same physical interpretation can be considered for the problem in any dimension $d$, by placing in 
$\bbR^{d+1}$ two close parallel walls (corresponding to the tube in dimension one), and placing the cells along them.

\subsubsection{Free energy functional}\label{s3:free-energy}

A tool allowing to show the convergence of trajectories of~\eqref{eq:flow} is the \emph{free energy} that, due to a general physical principle, should not increase along the trajectories as long as we do not do any work.

Namely, consider an absolutely continuous probability measure $\mu = \mu(x)\mathrm{d}x$
(by an abuse of notation, we denote the measure and its density by 
the same letter). Imagine $\mu(x)$ as the density of a gas, particles of
which implement the Brownian motion $\sqrt{2} \mathrm{d}B_t$, as
well as interact with the potential $W(x-y)$. 
Then, one defines the \emph{free energy} of $\mu$ as the sum of its ``entropy'' $\mathcal{H}$ and ``potential energy'':
\begin{equation}\label{def:free-VW}
\mF(\mu) := \mathcal{H}(\mu) + \frac{1}{2} \int_{\Rn} \int_{\Rn} \mu(x) W(x-y) \mu(y)\, \mathrm{d}x\, \mathrm{d}y,
\end{equation}
where the entropy of the measure $\mu$ is  
\begin{equation}\label{eq:entropy}
\mathcal{H}(\mu) := \int_{\Rn} \mu(x) \log\mu(x) \mathrm{d}x.
\end{equation}
Then, as we have already said, a general physical principle says that, as we are doing no work on the system, the free energy should decrease, and the system should tend to its minimum. 

Indeed, the free energy  $\mF$ is a Lyapunov function for the flow~\eqref{eq:flow} (when it is defined, though it is defined only for measures that are absolutely continuous with respect to the Lebesgue measure, and otherwise, $\mF(\mu) = +\infty$). This can be seen by joining two statements: on one hand, the measure $\Pi(\mu)$ is (what corresponds to the same physical principle) the unique global minimum of a free energy
%To verify this, one can first consider a Brownian gas in a potential $V$, defining its free energy as  
$$
\mF_{V}(\mu) := \mathcal{H}(\mu) + \int_{\Rn} V(x) \mu(x)\, \mathrm{d}x,
$$
of a non-interacting Brownian motion in the exterior potential $V=W*\mu$ (see \S\ref{ss:physics} and Lemma~\ref{l:fixed-point} in \S\ref{s:free-energy}). The second is the inequality 
\begin{equation}\label{eq:deriv-free}
\partial_{m-\mu} \mF|_{\mu} \le \mF_{W*\mu}(m) - \mF_{W*\mu}(\mu),
\end{equation}
where $m = \Pi(\mu)$. On one hand, it can be easily seen by an explicit computation, noticing that the entropy part is convex. On the other hand, such a differentiation corresponds to replacing some small parts of the gas distributed with respect to the measure $\mu$ by the one distributed with respect to the measure $m$, and in the right-hand side we have the corresponding free energies of these small parts in the potential, generated by the main part of the gas.

Then, differentiating the function~$\mF$ along the trajectories of the flow~\eqref{eq:flow}, one finds for the solution~$\mu(\theta)$ 
$$
\frac{\mathrm{d}}{\mathrm{d} \theta} \mF(\mu(\theta)) \le \mF_{W*\mu(\theta)}(\Pi_{W*\mu(\theta)}(\mu(\theta))) - \mF_{W*\mu(\theta)}(\mu(\theta)) \le 0,
$$
with the equality if and only if $\mu(\theta)=\Pi(\mu(\theta))$. 

Finally (and we recall these arguments in \S\ref{s:discretization}), the fixed points of $\Pi$ are exactly the translation images of the density $\rho_\infty$, that is the centered global minimum of the functional $\mF$. So, roughly speaking, the function $\mF$ is the Lyapunov function of the flow~\eqref{eq:flow}. The words ``roughly speaking'' here refer to that these arguments are non-rigorous: we avoided showing that the flow is indeed well-defined, the free energy functional is defined only for absolutely continuous measures, etc. Though all of this serves well as a motivation to (rigorous) lemmas of free energy behaviour used in this paper.

We conclude this paragraph by indicating that for the dynamics in presence of an exterior potential $V$  (the case of Theorem~\ref{t:main-2}) one has to replace the free energy function by
$$
\mF_{V,W}(\mu) := \mathcal{H}(\mu) + \int_{\Rn} V(x) \mu(x)\, \mathrm{d}x \\
+ \frac{1}{2} \int_{\Rn} \int_{\Rn} \mu(x) W(x-y) \mu(y)\, \mathrm{d}x\, \mathrm{d}y.
$$
and, instead of $\mF_{W*\mu}$, consider $\mF_{V+W*\mu}$ for the energy of ``small parts''.

\subsubsection{Conclusion}

We are now ready to conclude the sketches of the proofs of Theorems~\ref{t:centered} and~\ref{t:centers} (as it was already mentioned, they immediately imply Theorem~\ref{t:main}).

Namely, we consider the discretized Euler-like evolution of the flow~\eqref{eq:9-3/4}, defined by the rule 
\begin{equation}\label{eq:discr}
\tilde{\mu}_{T_{n+1}} = \tilde{\mu}_{T_n} + \frac{\Delta T_{n}}{T_{n+1}} \, 
(\Pi(\tilde{\mu}_{T_n}) - \tilde{\mu}_{T_n}).
\end{equation}
For the measures $\tilde{\mu}_{T_n}$ defined by this procedure, we obtain (using discrete rigorous analogues of informal arguments of the previous paragraph) some estimates on the speed with which their  free energy decreases. This allows us to estimate distances from these measures to the set of translates of $\rho_{\infty}$ (because they are the only minima of~$\mF$).

Now, we are taking the true random trajectory~$\mu_t$, and estimate the distance from the centered measures~$\mu_t^c$ to the equilibrium point. To do this  at some moment $t$, we choose an earlier moment $t'$, replace the measure $\mu_{t'}$ by a close smooth measure $\tilde{\mu}_{t'}$, and consider deterministic discrete iterates by~\eqref{eq:discr}. On one hand, for this new trajectory the free energy is defined (as we have chosen a smooth approximation), so we control the decrease of energy and hence the distance from the centered measure $\tilde{\mu}^c_t$ to $\rho_{\infty}$. On the other hand, an accurate computation allows us to control the distance between the random measure $\mu_t$ and the approximating deterministic image $\tilde{\mu}_t$ of its smooth perturbation. The sum of these distances then estimates the distance from $\mu_t^c$ to $\rho_{\infty}$, and the obtained estimate tends to 0 as $t\to\infty$. This concludes the proof of Theorem~\ref{t:centered}.

Finally, to prove Theorem~\ref{t:centers}, one first computes the speed of drift of the center~$c_t$, and then shows that the series of general term $|c_{T_{n+1}}-c_{T_n}|$ converges, and the oscillations $osc_{[T_n,T_{n+1}]} c_t$ tend to zero. This implies the existence of the limit of~$c_t$ as~$t\to\infty$.

\subsection{Outline of the paper}
At the beginning of Section~\ref{s:preliminaries}, we show the existence and uniqueness of solutions to~\eqref{eq:AABM2} starting at any positive moment $r>0$. The discussion of this topic at $t=0$ is postponed to the appendix. In the rest of Section~\ref{s:preliminaries}, we present some crucial preliminary computations which are at the basis of our proofs. Most of the material there is not new, except for the combination of stochastic approximation of the empirical measure (see \cite{BLR}) with free energy functionals (see \cite{CMV}) and the achieving of a bound on the convergence rate. Finally, Section~\ref{s:proof} is devoted to the proofs of our main results.

\subsection{Acknowledgments}
The authors are very grateful to two anonymous referees for their useful comments which led to a rewritting of the paper for a better understanding.

\section{Preliminaries}\label{s:preliminaries}

As usual, we denote by $\mathcal{M}(\Rn)$ the space of signed
(bounded) Borel measures on $\Rn$ and by $\mathcal{P}(\Rn)$ its
subspace of probability measures. We will need the following measure
space:
\begin{equation}
\mathcal{M}(\Rn;P):= \{\mu \in \mathcal{M}(\Rn); \int_{\Rn} P(|y|) \,
\vert\mu\vert(\mathrm{d}y)<\infty\},
\end{equation}
where $\vert\mu\vert$ is the variation of $\mu$ (that is
$\vert\mu\vert := \mu^{+} + \mu^{-}$ with $(\mu^{+},\mu^{-})$ the
Hahn-Jordan decomposition of $\mu$: $\mu = \mu^+ -\mu^-$). Belonging to this space will enable us to
always check the integrability of $P$ (and therefore of $W$ and its
derivatives thanks to the domination
condition~\eqref{eq:domination}) with respect to the (random)
measures to be considered. We endow this space with the dual weighted supremum norm (or dual $P$-norm) defined for $\mu \in \mathcal{M}(\Rn;P)$ by
\begin{equation}
\vert \vert \mu\vert \vert_P := \sup_{\varphi\in \mathcal{C}(\Rn); \vert \varphi\vert
\leq P} \left\vert \int_{\Rn} \varphi\, \mathrm{d}\mu\right\vert = \int_{\Rn} P(|y|) \, |\mu|(\mathrm{d}y).
\end{equation}
We recall that $P(|x|)\ge 1$, so that $\|\mu\|_P\ge |\mu(\Rn)|$.
This norm naturally arises in the approach to ergodic results for
time-continuous Markov processes of Meyn \& Tweedie \cite{meT}. It
also makes $\mathcal{M}(\Rn;P)$ a Banach space.

\noindent Next, we consider $\mathcal{P}(\Rn;P) =
\mathcal{M}(\Rn;P) \cap \mathcal{P}(\Rn)$. We remark that both $\mathcal{M}(\Rn;P)$ and $\mathcal{P}(\Rn;
P)$ contain any probability measure with an
exponential tail and, in particular, any compactly supported measure. For any $\kappa>0$, we also define 
\begin{equation}\label{eq:def-P-kappa}
\mathcal{P}_\kappa(\mathbb{R}^d;P):= \{ \mu \in \mathcal{P}(\mathbb{R}^d;P) \, ; \, \vert \vert \mu \vert\vert_P = \int_{\Rn} P(|x|) \mu (\mathrm{d}x) \leq \kappa\}.
\end{equation}

\subsection{Existence and uniqueness of solutions}\label{s:existence}

\subsubsection{Markovian form; local existence and uniqueness}\label{ss:Markov}

First step in studying the trajectories of~\eqref{eq:AABM2}
is to pass to the couple $(X_t,\mu_t)$. A standard remark is that
the behaviour of this couple is infinite-dimensional Markovian (and in general, except for $W$ being polynomial, cannot be reduced to a finite-dimensional Markov process).
 This reduction is easily implied by the
identity
\begin{equation}\label{eq:mu-t+s}
\mu_{t+s}=\frac{t}{t+s} \mu_t + \frac{1}{t+s} \int_t^{t+s}
\delta_{X_u} \, \mathrm{d}u.
\end{equation}
Note that the second term in the right-hand side 
of~\eqref{eq:mu-t+s} can be written as $\frac{s}{t+s}
\mu_{[t,t+s]}$, where $\mu_{[t,t+s]}$ is the empirical
measure during the time interval~$[t,t+s]$:
$$
\mu_{[t_1,t_2]}:=\frac{1}{t_2-t_1} \int_{t_1}^{t_2} \delta_{X_u}
\, \mathrm{d}u.
$$

Now, passing $\mu_t$ to the left-hand side of~\eqref{eq:mu-t+s}, dividing
by~$s$ and passing to the limit as $s\to 0$, we obtain the
following SDE for the couple $(X_t, \mu_t)$:
\begin{equation}\label{eq:x-mu}
\left\{
\begin{array}{l}
\mathrm{d}X_t= \sqrt{2}\, \mathrm{d}B_t - \nabla W*\mu_t(X_t)\, \mathrm{d}t, \\
\dot{\mu}_t = \frac{1}{t} (-\mu_t+\delta_{X_t}).
\end{array}
\right.
\end{equation}

For any $t_0>0$, the local existence and uniqueness of solutions
to~\eqref{eq:x-mu}, in a neighbourhood of $t_0$, is implied by well-known arguments: see
Theorem~11.2 of~\cite{RW}.

However, in order to study the asymptotic behaviour of solutions
to~\eqref{eq:AABM2}, we should first show the global existence of
these solutions, in other words, that they do not explode in a
finite time. It will be done in~\S\ref{ss:c-estimates}.

Note also that the equation~\eqref{eq:x-mu} clearly has a
singularity at~$t=0$. To avoid this singularity, sometimes the
equation~\eqref{eq:x-mu} is considered with an initial condition
$(X_r,\mu_r)$ at some positive time $r>0$ (and thus for $t\in
[r,\infty)$). After the time-shift $s=t-r$, the
system~\eqref{eq:x-mu} transforms to
\begin{equation}\label{eq:x-mu-r}
\left\{
\begin{array}{l}
\mathrm{d}X_s= \sqrt{2}\, \mathrm{d}B_s - \nabla W*\mu_s(X_s)\, \mathrm{d}s, \\
\dot{\mu}_s = \frac{1}{s+r} (-\mu_s +\delta_{X_s}).
\end{array}
\right.
\end{equation}

In fact, we can restrict our consideration to such situations only
(as, anyway, we are interested in the asymptotic behaviour of
solutions at infinity), but it is interesting to show that the
equation~\eqref{eq:AABM2} has indeed existence and uniqueness of
solutions for any initial value problem~$X_0=x_0$. It is done in the appendix.

\subsubsection{Center-drift estimates}\label{ss:c-estimates}

A natural ``reference point'' that one can associate to a
measure~$\mu$ is the equilibrium point $c_\mu=c(\mu)$ of the potential
it generates with~$W$, defined by the equation $\nabla W*\mu(c_\mu)=0$
(see Definition~\ref{def:mu-c}, \S\ref{ss:Main}), that we refer to
as the \emph{center} of the measure~$\mu$. Also, it will be
convenient to consider the \emph{centered measure} $\mu^c$,
obtained from~$\mu$ by the translation that shifts the center to
the origin.

Note that the implicit function theorem allows to estimate
(on an interval of existence of solution $(X_t,\mu_t)$
to~\eqref{eq:x-mu}) the derivative $\dot{c}_t$ of~$c_t:=c_{\mu_t}$.
In particular, we will see that $c_t$ is a $C^1$-function on this
interval.

Indeed, the function $(x,t) \mapsto \nabla W*\mu_t(x)$ is $C^1$-smooth:
\begin{eqnarray*}
\mathrm{d}(\nabla W*\mu_t)(x) &=& \nabla^2 W*\mu_t(x)\, \mathrm{d}x +
\nabla W*\dot{\mu}_t(x) \, \mathrm{d}t \\
&=& \nabla^2 W*\mu_t(x) \, \mathrm{d}x +
\frac{1}{t} \nabla W*(-\mu_t+\delta_{X_t})(x) \, \mathrm{d}t,
\end{eqnarray*}
and for any $(x,t)$ we have $\nabla^2 W*\mu_t(x)\ge C_W I >0$. The implicit function theorem
thus implies that $c_t$ is a $C^1$-function of $t$ (on the interval of existence of solution), and that
\begin{eqnarray*}
\dot{c}_t &=& -\left(\ddv{x} \nabla W*\mu_t(x) \left.\right|_{x=c_t}\right)^{-1} \ddv{t} (\nabla W*\mu_t)(c_t) = -\frac{1}{t} \left(\nabla^2 W*\mu_t(c_t)\right)^{-1} \nabla W*\delta_{X_t}(c_t) \\
&=& \frac{1}{t} \left(\nabla^2 W*\mu_t(c_t)\right)^{-1} \nabla W(X_t-c_t).
\end{eqnarray*}

This implies that the projection of the center drift velocity on the line from $c_t$ to $X_t$ is directed towards $X_t$, as $\nabla W(X_t-c_t)$ is positive, proportional to $X_t-c_t$ and  $$\left( \left(\nabla^2 W*\mu_t(c_t)\right)^{-1} \nabla W(X_t-c_t), X_t-c_t \right)>0.$$ This also immediately gives an upper bound on the drift speed:
\begin{equation}\label{eq:center-speed}
|\dot{c}_t| \le \frac{1}{t} \cdot \frac{P(|X_t-c_t|)}{C_W}.
\end{equation}

\subsubsection{Law of $X$-center distances: Ornstein-Uhlenbeck estimate}\label{ss:OU}

To continue our study, first we would like to obtain an estimate
on the behaviour of the distance $|X_t-c_t|$. Namely, we are going
to compare it with the (absolute value of) Ornstein-Uhlenbeck
process, and to obtain exponential-decrease bounds on its
occupation measure in \S\ref{ss:exp-dec}.

\begin{proposition}\label{p:OU-comparison}
The process $(X_t)$ can be considered as the first element of the pair $(X_t,Z_t)$ of processes such that
\begin{enumerate}
 \item[i)] $|X_t-c_t|\le 2 + Z_t$,
 \item[ii)] $Z_t$ is the absolute value of a $3d$-dimensional Ornstein-Uhlenbeck process.
\end{enumerate}
\end{proposition}
\begin{proof}
From $$
\left\{
\begin{array}{l}
\mathrm{d}X_t = \sqrt{2}\mathrm{d}B_t - \nabla W*\mu_t(X_t)\, \mathrm{d}t\\
\, \dot{c}_t = \frac{1}{t} \left(\nabla^2 W*\mu_t(c_t)\right)^{-1} \nabla W(X_t-c_t)
\end{array}
\right. $$
one obtains that the difference $|X_t-c_t|$, while it is non-zero, satisfies the SDE 
\begin{eqnarray*}
\mathrm{d}|X_t-c_t| &=& \sqrt{2}\left(\frac{X_t-c_t}{|X_t-c_t|}, \mathrm{d}B_t\right) + \frac{d-1}{|X_t-c_t|} \mathrm{d}t\\ 
&-& \left(\frac{X_t-c_t}{|X_t-c_t|}, \nabla W*\mu_t(X_t) + \frac{1}{t}(\nabla^2 W*\mu_t(c_t))^{-1}\nabla W(X_t-c_t) \right) \, \mathrm{d}t.
\end{eqnarray*}
In the same way, the desired $Z_t$ shall satisfy the equation 
\begin{equation}\label{eq:Z}
\mathrm{d}Z_t = \sqrt{2}\mathrm{d}\gamma_t - \left(\frac{C_W}{2}Z_t - \frac{3d-1}{Z_t}\right)\mathrm{d}t,
\end{equation}
where $\gamma$ is also a Brownian motion. 
So, take a one-dimensional standard Brownian motion $\beta$ independent of the Brownian motion $B$ and let $\gamma$ be defined as
\begin{equation}\label{eq:true-Z}
\mathrm{d}\gamma_t = \alpha(|X_t-c_t|) \left( \frac{X_t-c_t}{|X_t-c_t|}, \mathrm{d}B_t\right) + \sqrt{1-\alpha^2(|X_t-c_t|)}\mathrm{d}\beta_t ,
\end{equation}
 where $\alpha: [0,+\infty) \rightarrow [0,1]$ is a $C^\infty$-function which is identically zero in some neighbourhood of 0 and $\alpha(r)=1$ for any $r\ge 1$. The process $Z$ is then defined by~\eqref{eq:Z}.
 
We point out that, as $B$ and $\beta$ are independent, $B$ is a $d$-dimensional Brownian motion while $\beta$ is 1-dimensional. It follows that $Z$ defined by~\eqref{eq:Z} is the absolute value of a $3d$-dimensional Ornstein-Uhlenbeck process.

On the other hand, for any $t$, either $|X_t-c_t|\le 1$ and then automatically $|X_t -c_t|\le 2+Z_t$, or $|X_t-c_t|\ge 1$ and then both $|X_t-c_t|$ and $Z_t$ share exactly the same Brownian component (as $\alpha \equiv 1$), with the inequality between the drift terms of $2+Z_t$ and $|X_t-c_t|$: 
\begin{multline}
-\frac{C_W}{2}Z_t + \frac{3d-1}{Z_t} \ge  -C_W |X_t-c_t| +\frac{d-1}{|X_t-c_t|}\ge\\ 
\ge  - \left(\nabla W*\mu_t\left|_{X_t}\right., \frac{X_t-c_t}{|X_t-c_t|}\right) +\frac{d-1}{|X_t-c_t|}-\\
- \left(\frac{1}{t}(\nabla^2 W*\mu_t(c_t))^{-1}\nabla W(X_t-c_t), \frac{X_t-c_t}{|X_t-c_t|}\right),
\end{multline}
 as soon as $|X_t-c_t|\ge \frac{d-1}{3d-1}Z_t$. A comparison theorem concludes the proof.
\end{proof}

\subsubsection{Global existence}\label{ss:global-exist}

\begin{proposition}[global existence]\label{l:MBAA-global}
For any $r>0$ and for any initial condition $(X_r,\mu_r)$, the
solution to~\eqref{eq:x-mu} exists (and is unique) on the whole
interval~$[r,+\infty)$.
\end{proposition}
\begin{proof}
As we already have the local existence and uniqueness, it suffices
to check that the solution $X_t$ cannot explode in a finite time
(this impossibility will imply that the measures $\mu_t$, as the
occupation measures of $X_t$, also stay in a compact domain for
any bounded interval of time). 

Let us introduce the increasing sequence of stopping times $\tau_0 = 0$ and $$\tau_n :=
\inf\left\{t\geq \tau_{n-1}: \vert X_t\vert > n\right\}.$$
In order to show that the solution never explodes, we use the comparison of $X_t-c_t$ with the Ornstein-Uhlenbeck process $Z_t$ (see \S\ref{ss:OU}). So, we have for the corresponding $Z$, that $$\vert X_{\min (t, \tau_n)} -c_{\min (t, \tau_n)}\vert \leq 2+Z_{\min (t, \tau_n)}.$$ As $Z$ does not explode in a finite time, letting $n$ go to infinity, we conclude that $X_t-c_t$ does not explode in a finite time. To conclude, one has to use the inequality \eqref{eq:center-speed}: $$\vert \dot{c}_{t} \vert \le \frac{1}{t}\frac{P(\vert X_t -c_t \vert)}{C_W} \le \frac{1}{t}\frac{P(2+Z_t)}{C_W} \le \frac{1}{t}\frac{P(2)}{C_W}P(Z_t).$$ 
Any trajectory of $Z$ being bounded on any finite interval of time, the integral $\int_r^t \frac{P( Z_s )}{s} \mathrm{d}s$ is finite for any $t\ge r$. So, the process $(X_t,t\geq 0)$ does not explode in a finite time and there exists a global strong solution.
\end{proof}

\subsection{Exponential tails estimates}\label{ss:measures}

\subsubsection{Estimates for the centered empirical measure}\label{ss:exp-dec}
We shall now estimate the behaviour of the centered measures $\mu_t^c$. Namely, we are going to prove that these measures are exponentially decreasing. For shortness and simplicity, we introduce the following
\begin{definition}\label{eq:K-alpha-C}
Let $\alpha,C>0$ be given. Then
\begin{eqnarray}
K_{\alpha,C}^0 &:=& \{\mu \in \mcP(\Rn); \quad \forall r, \, \mu(\{y; |y|>r\})<Ce^{-\alpha r}\},\\
K_{\alpha,C} &:=& \{\mu \in \mcP(\Rn); \quad \mu^c \in K_{\alpha,C}^0 \}.
\end{eqnarray}
Also, for non-probability positive definite measures, we denote the spaces defined by the same inequalities by $\tilde{K}_{\alpha,C}^0$ and $\tilde{K}_{\alpha,C}$.
\end{definition}

For what follows, we need one easy lemma.
\begin{lemma}\label{l:OU_maj}
Let $Z$ be the absolute value of a $3d$-dimensional Ornstein-Uhlenbeck process. Then, there exists $C_1>0$, such that for almost any trajectory $Z_t$,
one has almost surely
$$\exists T: \forall t\ge T,\, \forall r>0 \quad \frac{1}{t}
\left| \{s\le t: Z_s>r \} \right|  < C_1 e^{-r}.$$
\end{lemma}
\begin{proof}
Note that the Ornstein-Uhlenbeck process is ergodic, with stationary measure $\gamma_{OU} = e^{-C_W |x|^2 /2}$. The function $f(x) = e^{|x|}$ is $\gamma_{OU}$-integrable. Hence, by (Birkhoff) ergodic theorem, almost surely $$\frac{1}{t}\int_0^t f(Z_s) \mathrm{d}s \rightarrow \int f(x) \mathrm{d}\gamma_{OU}(x) =: I.$$ Thus for all $t$ large enough, $\frac{1}{t}\int_0^t e^{|Z_s|} \mathrm{d}s \le I+1$. Applying Chebychev's inequality, we see that for all $r>0$, $$\frac{1}{t} \left| \{s\le t: Z_s>r \} \right| < (I+1) e^{- r}.\qedhere$$
\end{proof}

The main result of this subsection is the following, showing that the measure $\mu_t$ belongs to the set $K_{\alpha,C}$.
\begin{proposition}\label{p:expo-decrease}
There exist two constants $\alpha,C >0$ such that a.s. at any
sufficiently large time $t$, $\mu_t\in K_{\alpha,C}$.
\end{proposition}
To prove this proposition, we need two intermediate lemmas, which proofs are postponed.
\begin{lemma}\label{l:K-alpha0}
There exist $\alpha_0,C_0>0$ such that a.s. for any sufficiently large time $t$, $\mu_{[t/2,t]}(\cdot + c_{t/2}) \in K_{\alpha_0,C_0}^0$.
\end{lemma}
\begin{lemma}\label{l:sep-measure}
Let $\alpha_0,C_0>0$ be fixed. Then there exist $\alpha,C,C'$ such that the following holds. Assume that there are given a measure $\mu \in \mcP(\Rn;P)$, a measure $\nu(\cdot + c_\mu) \in K_{\alpha_0,C_0}^0$ and a coefficient $0<\lambda <1/2$. Then if $\mu$ can be decomposed as $\mu = \mu^{(1)} +\mu^{(2)}$ with $\mu^{(2)}(\cdot +c_\mu) \in \tilde{K}_{\alpha,C}^0$, then for the decomposition $(1-\lambda)\mu + \lambda \nu(\cdot + c_{\tilde{\mu}}) = (1-\lambda) \mu^{(1)} + \left( (1-\lambda) \mu^{(2)} + \lambda \nu(\cdot + c_{\tilde{\mu}})\right)$, one also has $\left( (1-\lambda) \mu^{(2)} + \lambda \nu(\cdot + c_{\tilde{\mu}})\right) \in \tilde{K}_{\alpha,C}^0$. 
\end{lemma}
In other words, this lemma provides an ``induction step" for showing that ``a big part of the centered measure has exponentially small tails" for a procedure of repetitive mixing with measure with exponential tails (this not being obvious, as the center can be shifted by such a procedure).

\begin{proof}[Proof of Proposition~\ref{p:expo-decrease}]
First, let us estimate the drift of the center. Namely, taking together~\eqref{eq:center-speed} and Proposition~\ref{p:OU-comparison},
we have
$$
|\dot{c}_t| \le \frac{1}{tC_W} P(|X_t-c_t|)\le \frac{1}{tC_W} P(2+Z_t)\le \frac{P(2)}{tC_W} P(Z_t),
$$
for the corresponding Ornstein-Uhlenbeck trajectory~$Z_t$.

On the other hand, $Z$ is a Harris recurrent process and $P(Z)$ is integrable with respect to the
Gaussian measure, thus due to the limit-quotient (or Birkhoff)
theorem, almost surely there exists a limit
$$
\lim_{t\to\infty} \frac{1}{t}\int_0^t P(Z_s) \, \mathrm{d}s =
\int_{\bbR^d} P(|z|) \, \mathrm{d}\gamma_{OU}(z) =: I.
$$
So, almost surely from some time $t_1$ we have
$$
\forall t>t_1, \quad  \frac{1}{t}\int_0^t P(Z_s) \, \mathrm{d}s
\leq I+1.
$$
Therefore, after this time we can estimate the displacement of
the center between the moments $t/2$ and $t$: $\forall t>t_1 $
\begin{eqnarray*}
|c_{t/2}-c_t| \le  \int_{t/2}^{t} |\dot{c}_s|
\, \mathrm{d}s \le \int_{t/2}^{t} \frac{C}{s} P(Z_s) \, \mathrm{d}s 
\le \frac{C}{t/2}\int_{0}^{t} P(Z_s) \, \mathrm{d}s \le
2C(I+1) =:C_3.
\end{eqnarray*}
In fact, the same estimate holds for any $t'$ between $t/2$ and
$t$:
$$|c_{t'}-c_t|\le C_3.
$$
This immediately implies that for any $t>t_1$ and $n\in\bbN$ such
that $2^{-n+1}t>t_1$, one has
$$
|c_t-c_{t/2^n}| \le C_3 n.
$$

Let us now apply Lemma~\ref{l:sep-measure}. First let us decompose, for any $t\in [t_1,2t_1]$, the measure $\mu_{2t}$ as $\frac{1}{2}\mu_t + \frac{1}{2} \mu_{[t,2t]}$, then the measure $\mu_{4t}$ as $\frac{1}{4}\mu_t + \left(\frac{1}{4} \mu_{[t,2t]} + \frac{1}{2}\mu_{[2t,4t]}\right)$, $\ldots$, and finally the measure $\mu_{2^nt}$ as $\frac{1}{2^n}\mu_t + \left(\frac{1}{2^n} \mu_{[t,2t]} +\cdots + \frac{1}{2}\mu_{[2^{n-1}t,2^n t]}\right)$. 
An induction argument, together with Lemma~\ref{l:K-alpha0}, immediately shows that in each such decomposition, the second term shifted by the corresponding $c(\mu_{2^j t})$ belongs to $\tilde{K}_{\alpha,C}^0$. The only part that is left to handle is $\frac{1}{2^n}\mu_t$. But the distance between $c_t$ and $c_{2^n t}$ does not exceed $C_3n$, and the centered measure $\mu_t^c$ is compactly supported. So it is contained in a ball of some (random) radius $R$ that can be chosen uniform over $t\in (t_1,2t_1)$. Now the measure $\frac{1}{2^n}\mu_t$ is of total weight $2^{-n}$ and it vanishes outside a radius $R$ ball. If $\alpha$ is small enough so that $e^{\alpha C_3}<2$, then for any $r>C_3n+R$, we have $$\frac{1}{2^n}\mu_t(|y- c_{2^n t}|>r)\le \frac{1}{2^n}\mu_t^c(|y|>r-C_3n) = 0,$$
and for $r\le C_3n +R$ and $n$ big enough,
$$ \frac{1}{2^n}\mu_t(|y- c_{2^n t}|>r)\le 2^{-n}< e^{-n\alpha C_3} e^{-\alpha R} \le e^{-\alpha r}.$$
The middle inequality comes, for $n$ large enough, from a comparison between exponent bases, $e^{\alpha C_3}<2$, with respect to which a multiplication constant $e^{-\alpha R}$ is minor. Finally, joining the obtained $\frac{1}{2^n}\mu_t(\cdot + c_{2^n t}) \in \tilde{K}_{\alpha,1}^0$ and $\left(\frac{1}{2^n} \mu_{[t,2t]} +\cdots + \frac{1}{2}\mu_{[2^{n-1}t,2^n t]}\right)(\cdot + c_{2^n t}) \in \tilde{K}_{\alpha,C}^0$, we obtain $\mu_{2^n t}\in K_{\alpha,C+1}$.
\end{proof}

\begin{proof}[Proof of Lemma \ref{l:K-alpha0}]
This lemma immediately follows from Lemma~\ref{l:OU_maj}, once we notice that 
\begin{eqnarray*}
\mu_{[t/2,t]}(|y-c_{t/2}|>r) &=& \frac{2}{t}\left| \{s: \, t/2 <s<t, \, |X_s-c_{t/2}|>r\}\right|\\ 
&\le & \frac{2}{t}\left| \{s: \, t/2 <s<t, \, |X_s-c_s|>r-|c_{t/2}-c_s|\}\right|\\ 
&\le &  \frac{2}{t}\left| \{s: \, s<t, \, Z_s>r-C_3\}\right|\le C_0 e^{\alpha_0 C_3}\cdot e^{-\alpha_0 r}.\qedhere 
\end{eqnarray*}
\end{proof}

\begin{proof}[Proof of Lemma \ref{l:sep-measure}]
First, let us estimate the position of the center of $\tilde{\mu}$ in a way that is linear in $\lambda$ and does not depend on $\alpha$ and $C$ --- thus in particular, proving the statement~i). Indeed, $c_{\tilde{\mu}}$ is the minimum of the function $W*\tilde{\mu}$. At the point $c_\mu$, the gradient of this function can be bounded as $$\left|\nabla W* \tilde{\mu}|_{c_\mu}\right| = \left| (1-\lambda)\nabla W*\mu|_{c_\mu} + \lambda \nabla W*\nu|_{c_\mu} \right| \le \lambda \|\nu(\cdot + c_\mu) \|_P \le C' \lambda,$$ because the norm $\|\nu(\cdot + c_\mu) \|_P$ is uniformly bounded due to the condition $\nu(\cdot + c_\mu) \in K_{\alpha_0,C_0}$.

Now, restricting the function $W*\tilde{\mu}$ on the line joining $c_\mu$ and $c_{\tilde{\mu}}$, that is considering $$f(s) = W*\tilde{\mu} \left(c_\mu + s \frac{c_{\tilde{\mu}} - c_\mu}{|c_{\tilde{\mu}} - c_\mu|}\right),$$ one sees that $|f'(0)|\le C' \lambda$, $f'(|c_{\tilde{\mu}} - c_\mu|)= 0$, $f''\ge C_W$, what implies $|c_{\tilde{\mu}} - c_\mu|\le \frac{C'}{C_W}\lambda$.

Let us now estimate the measure $\left( (1-\lambda)\mu^{(2)} + \lambda \nu \right)(|y-c_{\tilde{\mu}}|\ge r)$. Indeed, note that $\{y: \, |y-c_{\tilde{\mu}}| \ge r\}\subset \{y: \, |y-c_{\mu}| \ge r-C''\lambda\}$. Thus 
\begin{eqnarray}\label{eq:hat-mu}
\tilde{\mu}(|y-c(\tilde{\mu})|\ge r) &\le & \tilde{\mu}(|y-c(\mu)|\ge r-C''\lambda)\notag\\ 
&\le & (1-\lambda) \mu^{(2)}(|y-c(\mu)|\ge r-C''\lambda) + \lambda \Pi(\mu^c)(|y|\ge r-C''\lambda)\notag\\
&\le & (1-\lambda)C e^{-\alpha (r-C''\lambda)} + \lambda C_0 e^{-\alpha_0(r-C''\lambda)}\notag\\
&\le & \left( 1-\frac{\lambda}{2}\right)C e^{C''\alpha \lambda-\alpha r} - \lambda \left(\frac{C}{2} e^{-\alpha r} -C_0 e^{\alpha_0 C''\lambda-\alpha_0 r}\right) \notag\\
&\le & e^{\lambda(C''\alpha - 1/2)} C e^{-\alpha r} - \lambda \left( \frac{C}{2} e^{(\alpha_0 -\alpha) r} - C_0 e^{\alpha_0 C''\lambda} \right) e^{-\alpha_0 r}.
\end{eqnarray}
Once $\alpha$ is small enough so that $C''\alpha <1/2$, $\alpha <\alpha_0$ and once $C$ is greater than $2C_0 e^{\alpha_0 C''}$, the right-hand side of~\eqref{eq:hat-mu} is not greater than $Ce^{-\alpha r}$, what concludes the proof.
\end{proof}

\subsubsection{Estimates for the centered measure $\Pi$}\label{ss:flow-existence}
\begin{lemma}\label{Pibound}
For any $\kappa>1$, the map $\Pi$ restricted to
$\mathcal{P}_\kappa(\mathbb{R}^d;P)$ is bounded and Lipschitz.
\end{lemma}
\begin{proof}
First, we need to show that $Z(\mu)$ is bounded from
below on $\mathcal{P}_\kappa(\mathbb{R}^d;P)$. For $\mu \in \mathcal{P}_\kappa(\mathbb{R}^d;P)$, the domination
condition~\eqref{eq:domination} implies that $W*\mu(x) \leq \vert \vert \mu \vert \vert_P P(|x|) \leq \kappa P(|x|)$. So we
have:
$$Z(\mu)=\int_{\mathbb{R}^d} e^{-W*\mu(x)} \mathrm{d}x \geq \int_{\mathbb{R}^d}
e^{-\kappa P(|x|)} \mathrm{d}x.$$ Now, using that $W*\mu(x) = \int W(x-y) \mu(\mathrm{d}y) \ge \frac{C_W}{2} \int |x-y|^2 \mu(\mathrm{d}y) \ge \frac{C_W}{2} \int \left(\frac{|x|^2}{4} -|y|^2\right) \mu(\mathrm{d}y) \ge \frac{C_W}{2} \left(\frac{|x|^2}{4} -\kappa \right)$, 
%\left(\frac{\vert x\vert^2}{4} - \kappa\right)$ as $P(|y|)\ge |y|^2$ (because $\vert x-y \vert^2 \ge  \vert x \vert^2/4 - \vert y \vert^2$) and $W(x)\geq C_W|x|^2 /2$, 
we hence have the
following bound for $\Pi(\mu)$:
\begin{equation}\label{Cbeta}
\vert \vert \Pi(\mu) \vert\vert_P \leq \left( \int_{\mathbb{R}^d} e^{-\kappa P(|x|)} \mathrm{d}x\right)^{-1}\cdot \int_{\mathbb{R}^d} P(|x|) e^{-\frac{C_W}{2}(|x|^2/4 -\kappa)}\mathrm{d}x =: C_\kappa.
\end{equation}
Note that $\Pi$ is $C^1$ on $\mathcal{P}(\mathbb{R}^d;P)$ endowed with the strong topology. As the set of probability measures has no interior point, we have to specify the meaning of $C^1$: there exists a continuous linear operator $D\Pi(\mu): \mathcal{M}_0(\mathbb{R}^d;P) \rightarrow \mathcal{M}_0(\mathbb{R}^d;P)$, continuously depending on $\mu$, such that $\|\Pi(\mu') - \Pi(\mu) - D\Pi(\mu) (\mu-\mu')\|_P = O(\|\mu -\mu'\|_P)$ provided that $\mu'\in \mathcal{P}(\mathbb{R}^d;P)$ and $\mu'$ converges toward $\mu$. Indeed, it is easy to see that 
\begin{eqnarray}\label{diffPi}
D\Pi(\mu)\cdot \nu &:=& -(W*\nu) \Pi(\mu) -\frac{DZ(\mu)\cdot \nu}{Z(\mu)^2} e^{-W*\mu} \notag\\
&=& -(W*\nu) \Pi(\mu) + \int_{\mathbb{R}^d} W*\nu(y)\frac{e^{-W*\mu(y)}}{Z(\mu)}\mathrm{d}y \, \frac{e^{-W*\mu}}{Z(\mu)} \notag\\
&=& -\left(W*\nu -\int_{\mathbb{R}^d} W*\nu(y)\Pi(\mu)(\mathrm{d}y)\right) \Pi(\mu).
\end{eqnarray}
Now, note that the norms $\|D\Pi\|$ are uniformly bounded for $\mu \in \mathcal{P}_\kappa(\mathbb{R}^d;P)$ (for any given $\kappa$). Indeed, fix $\nu \in \mathcal{M}_0(\mathbb{R}^d;P)$. Since $\vert W*\nu(x) \vert \leq \vert \vert \nu \vert\vert_P P(|x|)$, we find that
\begin{equation*}
\| D\Pi(\mu)\cdot \nu \|_P \leq (1+C_\kappa) \| \nu \|_P \int_{\mathbb{R}^d} P^2(|x|)\Pi(\mu)(\mathrm{d}x).
\end{equation*}
For $\mu \in \mathcal{P}_\kappa(\mathbb{R}^d;P)$, the same
computation used for the bound~\eqref{Cbeta} on the norm of $\Pi(\mu)$ enables to control
the last integral. Hence, we deduce a bound (call it $C'_\kappa$) on
the norm of the differential. Thus, $\Pi$ is Lipschitz as stated.
\end{proof}

We prove now the exponential decrease for the centered measure $\Pi(\mu)$.
\begin{proposition}\label{prop:exp-decrease}
There exists a positive constant $C_\Pi$ such that for all $\mu\in
\mcP(\mathbb{R};P)$, for all $R>0$, we have $\Pi(\mu)(|x-c_\mu|\ge R)\le
C_\Pi e^{-C_W R}$.
\end{proposition}
\begin{proof}
Note first that, imposing a condition $C_\Pi\ge e^{2C_W}$, we can restrict ourselves only on $R\ge 2$: for $R<2$, the estimate is obvious. The measure $\Pi(\mu)$ has the density $\frac{1}{Z(\mu)} e^{-W*\mu(x)}$. To avoid working with the normalization constant $Z(\mu)$, we will prove a stronger inequality, that is $$\Pi(\mu)(|x-c_\mu|\ge R)\le
C_\Pi e^{-C_W R}\cdot \Pi(\mu)(|x-c_\mu|\le 2),$$
which is equivalent to $$\int_{|x-c_\mu|\ge R} e^{-W*\mu(x)} \mathrm{d}x \le C_\Pi e^{-C_W R} \int_{|x-c_\mu|\le 2} e^{-W*\mu(x)} \mathrm{d}x.$$ 
Pass to the polar coordinates, centered at the center $c_\mu$: we want to prove that $$\int_{\mathbb{S}^{d-1}} \int_R^\infty e^{-W*\mu (c_\mu + \lambda v)}\lambda^{d-1}\mathrm{d}\lambda \mathrm{d}v \le C_\Pi e^{-C_W R} \int_{\mathbb{S}^{d-1}} \int_0^2 e^{-W*\mu (c_\mu + \lambda v)}\lambda^{d-1}\mathrm{d}\lambda \mathrm{d}v.$$
It suffices to prove such an inequality ``directionwise'': for all $v\in \mathbb{S}^{d-1}$, for all $R\ge 2$ 
$$\int_R^\infty e^{-W*\mu (c_\mu + \lambda v)}\lambda^{d-1}\mathrm{d}\lambda \le C_\Pi e^{-C_W R} \int_0^2 e^{-W*\mu (c_\mu + \lambda v)}\lambda^{d-1}\mathrm{d}\lambda.$$ But from the uniform convexity of $W$ and the definition of the center, the function $f(\lambda) = W*\mu (c_\mu + \lambda v)$ satisfies $f'(0) =0$ and $\forall r>0$, $f''(r)\ge C_W$. Hence, $f$ is monotone increasing on $[0,\infty)$, and in particular, 
\begin{equation}\label{eq:coord-polaires}
\int_0^2 e^{-f(\lambda)} \lambda^{d-1} \mathrm{d}\lambda \ge e^{-f(2)} \int_0^2 \lambda^{d-1}\mathrm{d}\lambda =: C_1 e^{-f(2)}.
\end{equation}
On the other hand, for all $\lambda \ge 2$, $f'(\lambda) \ge f'(2)\ge 2C_W$, and thus $f(\lambda)\ge 2C_W (\lambda -2) + f(2)$. Hence, 
\begin{equation}\label{eq:coord-polaires-2}
\int_R^\infty e^{-f(\lambda)} \lambda^{d-1} \mathrm{d}\lambda \le e^{-f(2)} \int_R^\infty \lambda^{d-1} e^{-2C_W(\lambda-2)} \mathrm{d}\lambda \le C_2 R^{d-1} e^{-2C_W R} \cdot e^{-f(2)} \le C_3 e^{-C_WR} \cdot e^{-f(2)}.
\end{equation}
Comparing~\eqref{eq:coord-polaires} and~\eqref{eq:coord-polaires-2}, we obtain the desired exponential decrease.
\end{proof}

\subsection{A new transport metric: $\mT_P$-metric}\label{ss:invariant}

Usually, to estimate the distance between two probability measures, one introduces the Wasserstein distance. Indeed, for $\mu_1,\mu_2 \in \mathcal{P}(\Rn;P)$, we define $$W^2_2(\mu_1,\mu_2) := \inf \{ \mathbb{E}(|\xi_1-\xi_2|^2)\},$$ where the infimum is taken over the random variables such that $\{$law of $\xi_1\}= \mu_1$ and $\{$law of $\xi_2\}= \mu_2$. In our setting, for a measure $\mu$, the corresponding probability
measure $\Pi(\mu)$ is defined using the convolution $W*\mu$. So,
it would be rather natural to use a distance, looking like the one
for the weak* topology, but allowing to control $W*\mu$ for our
unbounded function $W$. This motivates to introduce a new metric looking like the Wasserstein distance:
\begin{definition}
For $\mu_1,\mu_2 \in \mcP(\Rn;P)$, we define the \textit{$P$-translation
distance} between them as 
\begin{equation}
\mT_P(\mu_1,\mu_2) := \inf\left\{ \iint_0^1 P(|f(s,\omega)|) |f'_s(s,\omega)| \, \mathrm{d}s
\mathrm{d}\mathbb{P}\right\}, 
\end{equation}
where the infimum is taken over
 the maps $f:[0,1]\times \Omega
\rightarrow \mathbb{R}$, where $\Omega$ is a probability space, such that $\{\text{law of } f(0,\cdot) \} = \mu_1$,
and $\{\text{law of } f(1,\cdot)\} = \mu_2.$

We also denote the $\mT_P$-distance between two $c$-centered measures by $\mT_P^c(\mu_1,\mu_2)=\mT_P(\mu_1(\cdot +c),\mu_2(\cdot+c))$.
\end{definition}
\begin{remark}
In dimension one, we have the equivalent definition: 
$$\mT_P(\mu_1,\mu_2) := \int_{\mathbb{R}} P(|x|) |\mu_1((-\infty,x])-\mu_2((-\infty,x])|\, \mathrm{d}x.$$
\end{remark}

The following lemma will be useful to show the convergence of the empirical measure in the $W_2$-meaning, as Proposition~\ref{p:exp-conv} shows.
\begin{lemma}
Let $\mu_1,\mu_2\in \mathcal{P}(\Rn;P)$. There exists a constant $C>0$, such that $$W_2^2(\mu_1,\mu_2) \le C \mT_P(\mu_1,\mu_2).$$ If moreover $\mu_1, \mu_2$ belong to a set $K_{\alpha,C_0}$, then there exists $C'>0$, such that $$\mT_P(\mu_1,\mu_2)\le C' W_2^2(\mu_1,\mu_2).$$
\end{lemma}
\begin{proof}
Suppose that $\mu_1,\mu_2\in K_{\alpha,C_0}$. Take $\xi_1,\xi_2$ realizing the optimal $W_2$-transport between them, and let us estimate the $\mT_P$-cost of the same transport. Indeed, $$\mT_P(\mu_1,\mu_2) \le \int |\xi_1 - \xi_2|P(\max (|\xi_1|,|\xi_2|))\mathrm{d}F_{\max (\xi_1,\xi_2)}\le W_2^2(\mu_1,\mu_2) \int P^2(\max (|\xi_1|,|\xi_2|))\mathrm{d}F_{\max (\xi_1,\xi_2)},$$ where the second inequality is the Cauchy one. As $\mu_1,\mu_2\in K_{\alpha,C_0}$, we conclude that $$\int P^2(\max (|\xi_1|,|\xi_2|))\mathrm{d}F_{\max (\xi_1,\xi_2)}\le \int P^2(r) \mathrm{d} \max (0,1-2C_0 e^{-\alpha r})=: C'<\infty.$$

Let now $\xi_1,\xi_2$ be two random variables corresponding to the $\mT_P$-optimal transport of $\mu_1$ to $\mu_2$. We then have 
\begin{eqnarray}\label{eq:W-T_P}
W^2_2(\mu_1,\mu_2) = \int |\xi_1-\xi_2|^2 \mathrm{d}\mathbb{P} \le \int |\xi_1 - \xi_2|\cdot 2 \max (\xi_1,\xi_2) \le \int |\xi_1 - \xi_2|\cdot \frac{P(\max (\xi_1,\xi_2)/2)}{4}\notag \\ 
\le C \mT_P(\mu_1,\mu_2). 
\end{eqnarray}
Indeed, the inequality~\eqref{eq:W-T_P} is due to the fact that the path between $\xi_1$ and $\xi_2$ either stays outside the radius $\max (|\xi_1|,|\xi_2|)/2$-ball centered in 0, in which case we estimate its length from below as $|\xi_1-\xi_2|$, or this path has a part joining the maximum norm vector to this ball, which is of length greater than $\max (|\xi_1|,|\xi_2|)/2 \ge |\xi_1-\xi_2|/4$.
\end{proof}

It is clear from the definition that $\mT_P$ is a distance; and also
taking into account that $|P'|\le P$, one easily has
\begin{equation}
\|\mu_2\|_P \le \|\mu_1\|_P + \mT_P(\mu_1,\mu_2).
\end{equation}
Thus, the set $\mcP(\Rn;P)$ is $\mT_P$-complete. Indeed, a $\mT_P$-Cauchy
sequence $(\mu_n)$ will have a weak limit $\mu$ and it is easy to
check that $\|\mu\|_P = \underset{n\rightarrow \infty}{\lim}
\|\mu_n\|_P <\infty$. So, $\mu\in \mcP(\Rn;P)$. Now, we are going
to estimate the deviance of trajectories in terms of $\mT_P$-metric, a result that will be useful in \S\ref{s:discretization}.

\begin{lemma}\label{l:mT_P}
For $\mu_1,\mu_2 \in \mcP(\Rn;P)$, the following statements hold:
\begin{enumerate}
\item[1)] The map $c$ is locally Lipschitz in the sense of $\mT_P$-metric:
$$
|c(\mu_1)-c(\mu_2)|\le \frac{1}{C_W} \min (P(\vert c(\mu_1)\vert), P(\vert c(\mu_2)\vert)) \cdot \mT_P(\mu_1,\mu_2);
$$
\item[2)] For all $v\in \Rn$, we have $\mT_P(\mu,\mu(\cdot +v)) \le \vert v\vert P(\vert v\vert) \|\mu\|_P$;
\item[3)] There exists $C_P>0$ such that $$\mT_P^c (\mu,\nu) \le \sup_{x\ge 0} \frac{P(x+\vert v\vert)}{P(x)} \mT_P(\mu,\nu) \le \begin{cases} (1+C_P \vert v\vert) \cdot \mT_P(\mu,\nu), & \vert v\vert\le 1\\ P(\vert v\vert)\mT_P(\mu,\nu), & \forall \vert v\vert ; \end{cases}$$
\item[4)] For all $\kappa>0$, $\mu^c:\, \mathcal{P}_\kappa(\Rn;P) \to \mathcal{P}(\Rn;P)$ is $\mT_P$-Lipschitz.
\end{enumerate}
\end{lemma}
\begin{proof}
1) Denoting by $c_1$ (resp. $c_2$) the center of
$\mu_1$ (resp. $\mu_2$), we have 
$$\nabla W*\mu_2(c_1) = \nabla W*\mu_1(c_1) + \nabla W*(\mu_2-\mu_1)(c_1),$$
thus $|\nabla W*\mu_2(c_1)|\le P(|c_1|) ||\mu_2-\mu_1||_P$. Joining the points $c_1$ and $c_2$ by a line, recalling that due to the uniform convexity of $W$, the second derivative of $W*\mu_2$ along this line is at least $C_W \mu_2(\Rn)$ and noticing that $\nabla W*\mu_2(c_2) = 0$, we obtain 
\begin{equation}
 |c_2 - c_1| \le \frac{P(|c_1|)}{C_W} \|\mu_2 -\mu_1\|_P.
\end{equation}
Similarly, we have $|c_2 - c_1| \le \frac{P(|c_2|)}{C_W} \|\mu_2 -\mu_1\|_P$. So, the result follows as $\|\mu_2 -\mu_1\|_P\le \mT_P(\mu_1,\mu_2)$.

2) We have by definition of $\mT_P$ that 
$$\mT_P(\mu, \mu(\cdot + |v|)) = \int_{\Rn} \mu(\mathrm{d}x) \int_{|x-y|\le |v|} P(|y|) \mathrm{d}y \le \int_{\Rn}  |v| P(|x|+|v|) \mu(\mathrm{d}x) \le |v| P(|v|)\int_{\Rn}  P(|x|) \mu(\mathrm{d}x) .$$
  
3) For any transport $f(s,\omega)$ between $\mu =\{$law of $f(0,\omega)\}$ and $\nu=\{$law of $f(1,\omega)\}$, the map $f(s,\omega) -v$ is a transport between $\mu^c$ and $\nu^c$ of price  $$\int_\Omega \int_0^1 P(|f(s,\omega)-v|) |f'_s(s,\omega)| \mathrm{d}s \mathrm{d}P(\omega) \le \sup_{x\ge 0} \frac{P(x+|v|)}{P(x)} \int_\Omega \int_0^1 P(|f(s,\omega)|) |f'_s(s,\omega)| \mathrm{d}s \mathrm{d}P(\omega).$$
The left-hand side is an upper bound for $\mT_P^c(\mu,\nu)$ and passing in the right-hand side to the infimum over all the possible transports $f$, we obtain the desired $\sup_{x\ge 0} \frac{P(x+|v|)}{P(x)} \mT_P(\mu,\nu)$.
 
4) Suppose that $\mu,\nu \in \mathcal{P}_\kappa(\Rn;P)$. Then, by the preceding points, we have
\begin{eqnarray*}
\mT_P (\mu(\cdot + c_\mu),\nu(\cdot +c_\nu)) &\le & \mT_P (\mu(\cdot + c_\mu),\nu(\cdot +c_\mu)) + \mT_P (\nu(\cdot + c_\mu),\nu(\cdot +c_\nu))\\
&\le & P(\vert c_\mu\vert) \mT_P(\mu,\nu) + \vert c_\mu - c_\nu \vert P(\vert c_\mu - c_\nu \vert) \vert\vert \nu(\cdot +c_\nu) \vert\vert_P\\
&\le & P(\vert c_\mu\vert ) \mT_P(\mu,\nu)\\
&+&  P(\vert c_\mu - c_\nu \vert) \frac{1}{C_W} \min (P(\vert c_\mu\vert), P(\vert c_\nu\vert))\|\nu\|_P \mT_P(\mu,\nu).
\end{eqnarray*} 
Remark that, as $\mu, \nu \in \mathcal{P}_\kappa(\Rn;P)$, the norms $|c_\mu|$ and $|c_\nu|$ are uniformly bounded, as well as $\|\nu\|_P$, thus $$\mT_P (\mu(\cdot + c_\mu),\nu(\cdot +c_\nu)) \le  \left(P(\vert c_\mu\vert ) +  P(\vert c_\mu - c_\nu \vert) \frac{1}{C_W} \min (P(\vert c_\mu\vert), P(\vert c_\nu\vert))\|\nu\|_P\right)  \mT_P(\mu,\nu),$$ where $P(\vert c_\mu\vert ) +  P(\vert c_\mu - c_\nu \vert) \frac{1}{C_W} \min (P(\vert c_\mu\vert), P(\vert c_\nu\vert))\|\nu\|_P$ is uniformly bounded by some constant $C_\kappa$, which is the Lipschitz constant.
\end{proof}

\subsection{Free energy functional}\label{s:free-energy}

We recall from \S\ref{s3:free-energy} that the free energy of a measure is defined as $$\mF(\mu) = \mathcal{H}(\mu) + \frac{1}{2}\iint \mu(x) W(x-y) \mu(y) \, \mathrm{d}x \mathrm{d}y,\quad \mathcal{H}(\mu) = \int \mu(x) \log \mu(x) \mathrm{d}x.$$ 
The free energy of a non-self-interacting gas in an exterior potential $V$ is defined as $$\mF_V(\mu) = \mathcal{H}(\mu) + \int \mu(x) V(x) \mathrm{d}x$$ and the map $\Pi$ associates to a measure $\mu$ the probability measure $\frac{1}{Z} e^{-W*\mu(x)} \mathrm{d}x$ (when $W*\mu$ is well-defined).

The first auxiliary statement implies that, as we mentioned it in \S\ref{s3:free-energy}, $\Pi(\mu)$ is the unique global minimum of $\mF_{W*\mu}$.

\begin{lemma}\label{l:fixed-point}
For any potential $V$ such that $e^{-V}$ is integrable, the probability measure $Z^{-1} e^{-V}$ is the unique global minimum of $\mF_V$ on $\mcP(\Rn)$.
\end{lemma}
\begin{proof}
Let $\mu = Z^{-1} e^{-V}$. Then, for any arbitrary absolutely continuous measure $\nu$, letting $\rho(x) = Z e^{V(x)} \nu(x)$ be its density with respect to $\mu$, we see that 
\begin{equation*}
\mF_{V}(\nu)= \int_{\Rn} (V(x) + \log \nu(x)) \nu(\mathrm{d}x) = \int_{\Rn} (\log \rho(x) - \log Z) \nu(\mathrm{d}x) = \int_{\Rn} \rho(x) \log \rho(x) \mu(\mathrm{d}x) - \log Z,
\end{equation*} 
and thus Jensen's inequality, for the convex function $\rho \log \rho$, leads immediately to the conclusion.
\end{proof}

Now, for the free energy functional, McCann~\cite{MC} proved the following
\begin{proposition}[McCann]
There exists a centered symmetric density $\rho_\infty$, which is a unique, up to translation, global minimum of $\mF$. Moreover, $\mF$ is a displacement convex functional, that is for two probability measures $\mu_0,\mu_1$ and the Wasserstein-optimal transport between them $$\xi_s= (1-s) \xi_0 +s\xi_1,$$ where $\mu_0 = \{\text{law of} \ \xi_0\}$, $\mu_1 = \{\text{law of} \ \xi_1\}$, $\mathbb{E}|\xi_0-\xi_1|^2 = W_2^2(\mu_0,\mu_1)$, one has $$\mF(\{\text{law of}\ \xi_s\}) \ge (1-s) \mF(\mu_0) + s \mF(\mu_1).$$
Finally, the transport distance from a centered measure $\mu$ to $\rho_\infty$ can be estimated as 
\begin{equation}\label{eq:W2F}
W_2^2 (\mu,\rho_\infty) \le \frac{2}{C_W} \mF(\mu|\rho_\infty),
\end{equation}
where $\mF(\mu|\rho_\infty) = \mF(\mu) -\mF(\rho_\infty)$. 
\end{proposition}
\begin{remark}
i) The uniqueness of the minimum comes from the strict displacement convexity of the restriction to the space of centered measures.

ii) The functional $\mF$ is \textit{not} convex in the usual sense, due to the self-interacting part.

iii) Inequality~\eqref{eq:deriv-free} together with Lemma~\ref{l:fixed-point} immediately imply that the minimum of $\mF$ is also a fixed point of $\Pi$.
\end{remark}

Finally, as we are going to work in \S\ref{sss:dec-energy} with the discretized flow, we will need two auxiliary statements for the free energy:
\begin{lemma}\label{l:equality-varphi}
For all absolutely continuous measures $\mu,\nu\in \mcP(\Rn;P)$ of finite free energy and for all $\lambda \in [0,1]$, we have 
\begin{multline}\label{eq:lem}
\mF((1-\lambda)\mu + \lambda \nu|\rho_\infty) \le \mF(\mu|\rho_\infty) - \lambda (\varphi_\mu(\mu)-\varphi_\mu(\nu))+\\ 
+ \frac{\lambda^2}{2} \iint (\mu-\nu)(x) W(x-y) (\mu-\nu)(y) \, \mathrm{d}x \mathrm{d}y,% + \iint \nu(x) W(x-y) \nu(y) \, \mathrm{d}x \mathrm{d}y -\right.\\
%\left.  -2 \iint \mu(x) W(x-y) \nu(y) \, \mathrm{d}x \mathrm{d}y\right),
\end{multline} 
where $\varphi_\mu(\cdot) := \mF_{W*\mu}(\cdot)$ is the free energy in the $\mu$-generated potential.

Moreover, for all absolutely continuous $\mu\in \mcP(\Rn;P)$, we have 
\begin{equation}\label{eq:lem-varphi}
\varphi_\mu(\mu) -\varphi_\mu(\nu) = \mF(\mu) -\mF(\nu) + \frac{1}{2} \iint (\mu-\nu)(x) W(x-y) (\mu-\nu)(y) \, \mathrm{d}x \mathrm{d}y.
\end{equation}
\end{lemma}
\begin{proof}
Note that $\mathcal{H}((1-\lambda)\mu + \lambda \nu) \le (1-\lambda)\mathcal{H}(\mu) + \lambda \mathcal{H}(\nu) = \mathcal{H}(\mu) -\lambda \left(\mathcal{H}(\mu) - \mathcal{H}(\nu)\right)$. So, it suffices to prove~\eqref{eq:lem} with entropy terms removed form both sides (from both $\mF$ and $\varphi_\mu$ in the right-hand side). After this removing, the formula becomes a Taylor expansion for a degree two polynomial. The same holds for~\eqref{eq:lem-varphi}, with a remark that the entropy terms are exactly the same in both sides.
%We compute
% \begin{multline*}
%\mF((1-\lambda)\mu + \lambda \nu) \le (1-\lambda)\mathcal{H}(\mu) + \lambda \mathcal{H}(\nu) + \frac{(1-\lambda)^2}{2} \iint \mu(x) W(x-y) \mu(y) \, \mathrm{d}x \mathrm{d}y +\\ 
%+ \lambda(1-\lambda) \iint \mu(x) W(x-y) \nu(y) \, \mathrm{d}x \mathrm{d}y+ 
%\frac{\lambda^2}{2} \iint \nu(x) W(x-y) \nu(y) \, \mathrm{d}x \mathrm{d}y =\\
%= \mathcal{H}(\mu) + \frac{1}{2} \iint \mu(x) W(x-y) \mu(y) \, \mathrm{d}x \mathrm{d}y -\\ 
%- \lambda \left(\mathcal{H}(\mu) + \iint \mu(x) W(x-y) \mu(y) \, \mathrm{d}x \mathrm{d}y -\{\mathcal{H}(\nu) + \iint \mu(x) W(x-y) \nu(y) \, \mathrm{d}x \mathrm{d}y\} \right)+\\
%+\frac{\lambda^2}{2} \left( \iint \mu(x) W(x-y) \mu(y) \, \mathrm{d}x \mathrm{d}y + \iint \nu(x) W(x-y) \nu(y) \, \mathrm{d}x \mathrm{d}y -2 \iint \mu(x) W(x-y) \nu(y) \, \mathrm{d}x \mathrm{d}y\right) =\\
%= \mF(\mu) - \lambda (\varphi_\mu(\mu) -\varphi_\mu(\nu)) + \frac{\lambda^2}{2} \left( \iint \mu(x) W(x-y) \mu(y) \, \mathrm{d}x \mathrm{d}y +\right.\\
%\left. +\iint \nu(x) W(x-y) \nu(y) \, \mathrm{d}x \mathrm{d}y -2 \iint \mu(x) W(x-y) \nu(y) \, \mathrm{d}x \mathrm{d}y\right).
% \end{multline*}
%Substracting $\mF(\rho_\infty)$ from the both sides, we obtain the desired formula.
\end{proof}

\section{Proofs}\label{s:proof}

\subsection{Proof of Theorem~\ref{t:centered}}\label{s:discretization}

In fact, we will prove a stronger statement, controlling the speed of convergence in the sense of the transport distance:

\begin{proposition}\label{p:exp-conv}
There exists $a>0$ such that almost surely, as $t\to\infty$, 
$$
\mT_P(\mu^c_{t},\rho_\infty) = O\left( e^{-a\sqrt[k+1]{\log t}} \right),
$$ 
where $k$ is the degree of the polynomial $P$, as well as 
$$
W_2(\mu^c_{t},\rho_\infty) = O\left( e^{-a\sqrt[k+1]{\log t}} \right).
$$
\end{proposition}

The proof of this statement will be decomposed into several propositions. We first present them all, postponing their proofs; then deduce from them Proposition~\ref{p:exp-conv}. Finally, we prove these propositions.

In order to prove this statement, as it was announced in \S\ref{ss:outline}, we will discretize the random process. Namely, we define the sequence $T_n$ of moments of time as $T_n:=n^{3/2}$ and then, $\Delta T_n:=T_{n+1}-T_n$ is of order $T_n^{1/3}$. Also, for what follows, we will associate to a random trajectory $(X_t,t\ge 0)$ thz sequence $(L_n)$ defined as
\begin{equation}\label{eq:def-L_n}
L_n:=\max\limits_{0\le t \leq T_{n+1}} |X_t -c_{T_n}|\le C_3'\log T_n.
\end{equation}
An easy conclusion from the Ornstein-Uhlenbeck comparison~\S\ref{ss:OU} and logarithmic drift of the center is that almost surely $L_n \le C_3' \log n$ and $L_n' \le \log n$ for any $n$ large enough.

Now, let us state the first of the propositions mentioned above, the one allowing to estimate the ``Euler-method'' one-step error in the description of the behaviour of measures $\mu_t$:

\begin{proposition}\label{p:one-step}
Almost surely there exists $n_0$ such that 
for any $n\ge n_0$, we have
$$
\mT_P^{c_{T_n}} (\mu_{[T_n, T_{n+1}]}, \Pi(\mu_{T_n})) \le (\Delta T_n)^{-\beta},
$$
where $\beta=\min \left(8C_W,\frac{1}{5d} \right) $.
\end{proposition}

Associated to the moments of time $T_n$, consider the following, roughly speaking, Euler-approximation maps for the flow $\dot{m}=\frac{1}{t}(\Pi(m)-m)$, with the knots chosen at the 
moments~$T_n$:
\begin{definition}\label{def:Phi}
For any $i\le j$, define $\Phi_i^j : \mcP(\Rn,P) \to \mcP(\Rn,P)$ as
$$
\Phi_i^i = id, \quad \Phi_i^{i+1}(\mu) = \mu + \frac{\Delta T_i}{T_{i+1}} (\Pi(\mu)-\mu), \quad \Phi_i^j= \Phi_{j-1}^j \circ \dots \circ \Phi_i^{i+1}.
$$
\end{definition}
Let us first exhibit an invariant set for $\Phi$.

\begin{lemma}\label{l:K-alpha-C}
For any $\alpha,C$ as in Lemma~\ref{l:sep-measure}, corresponding to $\alpha_0=C_W$ and $C_0=C_\Pi$ (from Proposition~\ref{prop:exp-decrease}), if $\mu\in K_{\alpha,C}$ and $i\le j$, then $\Phi_i^j(\mu) \in K_{\alpha,C}$.
\end{lemma}
\begin{proof}
This is a direct corollary of Lemma~\ref{l:sep-measure}.
\end{proof}

Denote, for a probability measure $\mu$ and for a number $h>0$, by $\mu^{(h)}$ the ``smoothening convolution''
$$
\mu^{(h)}:=\mu *  \left(\frac{1}{vol(\mathcal{U}_h(0))} \cdot  \1_{\mathcal{U}_h(0)} \, \mathrm{d}x \right),
$$
where $\mathcal{U}_h(0)$ is the radius $h$ ball in~$\Rn$, centered at the origin.

The following proposition allows to compare the deterministic Euler-like behaviour of the smoothened, at some moment $T_i$, measure with the true random trajectory:

\begin{proposition}\label{p:Euler-smooth}
There exist some constants $A, C_1, C_2, C_3>0$ such that almost surely there exists $n_0$ for which the following statements hold. For any $j> i \ge n_0$ and any $h>0$, 
\begin{equation}\label{eq:sum}
\mT_P^{c_{T_j}} (\Phi_i^j(\mu_{T_i}^{(h)}), \mu_{T_j}) \le \sum_{k=i}^{j-1} \frac{\Delta T_k}{T_{k+1}} (\Delta T_{k})^{-\beta} \left(\frac{T_j}{T_k} \right)^A + C_1 h \left(\frac{T_j}{T_i}\right)^A,
\end{equation}
provided that the right-hand side of~\eqref{eq:sum} does not exceed~$C_3$. Also, under the same condition,
$$
| c(\Phi_i^j(\mu_{T_i}^{(h)})) - c_{T_j}) | \le C_2.
$$
\end{proposition}

Next, we have to show that the deterministic trajectory of an absolutely continuous measure sufficiently fast approaches the set of translates of $\rho_{\infty}$. To do this, due to the estimate~\eqref{eq:W2F}, it suffices to estimate the free energy:
\begin{proposition}\label{p:exp-decrease}
Let $\mu\in K_{\alpha, C}$. Then, there exist $a_1,C_4,C_5>0$ such that almost surely there exists $n_0$ for which the following statements hold for any $j\ge i\ge n_0$:
\begin{enumerate}
 \item[i)] $\mF(\Phi_i^j(\mu) | \rho_{\infty}) \le C_4 + \frac{T_i}{T_j} (\mF(\mu|\rho_\infty) - C_4)$,
 \item[ii)] $\mF(\Phi_i^j(\mu) | \rho_{\infty}) \le C_5 e^{-a_1\sqrt[k+1]{\log (T_i/T_j)}}$ if $\mF(\mu|\rho_\infty)\le 2C_4$.
\end{enumerate}
\end{proposition}

Now, modulo these propositions, we are ready to prove Proposition~\ref{p:exp-conv}.

\begin{proof}[Proof of Proposition \ref{p:exp-conv}]
Recall from Proposition~\ref{p:one-step} that $\beta = \min (8C_W,(5d)^{-1})$. Note first that the distances $\mT_P^{c_{T_n}} (\mu_t,\mu_{T_n})$ for $t\in [T_n,T_{n+1}]$ are uniformly bounded for $n$ sufficiently big by 
$$
\frac{L_n^{k+1}  \Delta T_n}{T_{n+1}} \le c\frac{(\log n)^{k+1}}{n} \ll e^{-\sqrt[k+1]{\log n}};
$$
where $L_n$ is defined by~\eqref{eq:def-L_n}. Hence, it suffices to check the estimate for the subsequence of moments $T_n$:
$$
\mT_P (\mu_{T_n}^c, \rho_{\infty}) \le e^{-a \sqrt[k+1]{\log n}}.
$$
Now, for any sufficiently big $n$, take $i:=[n^{1-\delta}]$, where a small $\delta>0$ will be chosen and fixed (in a way that does not depend on~$n$) later. Then, considering for some $h>0$ a smoothened convolution $\mu_{T_i}^{(h)}$ and its Euler-image $\Phi_i^n(\mu_{T_i}^{(h)})$, we have by Proposition~\ref{p:Euler-smooth}
\begin{equation}\label{eq:star-E}
\mT_P^{c_{T_n}} (\mu_{T_n}, \Phi_i^n(\mu_{T_i}^{(h)})) \le \sum_{k=i}^{n-1} \frac{(\Delta T_k)^{1-\beta}}{T_{k+1}} \left(\frac{T_n}{T_k}\right)^A + C_1 h \left( \frac{T_n}{T_i} \right)^A,
\end{equation}
provided that the right-hand side does not exceed~$C_3$.

Denote by $\const$ a generic constant. Let us estimate the first term in the right-hand side:
\begin{multline}
\sum_{k=i}^{n-1} %\underbrace{\frac{\Delta T_k}{T_{k+1}}}_{\le 1/k} (\Delta T_k)^{-\beta} \left( \frac{T_n}{T_k} \right)^A \le
\frac{\Delta T_k}{T_{k+1}} (\Delta T_k)^{-\beta} \left( \frac{T_n}{T_k} \right)^A \le
%\sum_{k=i}^{n-1} \frac{1}{i} (\Delta T_i)^{-\beta} \left( \frac{T_n}{T_i} \right)^A \le
\sum_{k=i}^{n-1} \frac{2}{k} (\Delta T_k)^{-\beta} \left( \frac{T_n}{T_k} \right)^A \le \\ 
\le \sum_{k=i}^{n-1} \frac{2}{i} (\Delta T_i)^{-\beta} \left( \frac{T_n}{T_i} \right)^A \le 
\const \cdot n  \frac{(i^{1/2})^{-\beta}}{i} \left( \frac{n^{3/2}}{i^{3/2}} \right)^A  \le \const \cdot n^{(1+ A+\frac{\beta}{2})\delta-\frac{\beta}{2}}.
\end{multline}
So, for any fixed choice of $\delta<\frac{\beta/2}{1+A+(\beta/2)}$, the first term in the right-hand side of~\eqref{eq:star-E} will decrease as a negative power of $n$ and thus quicker than $e^{-a\sqrt[k+1]{\log T_n}}$.

Take now $h=\frac{C_3}{C_1} \left(\frac{T_i}{T_n} \right)^{A+1}$. For such a choice of $h$, the second term in the right-hand side of~\eqref{eq:star-E} is not greater than $\frac{T_i}{T_n}\sim n^{-\delta}$. So it also decreases quicker than $e^{-a\sqrt[k+1]{\log T_n}}$ and thus $\mT_P^{c_{T_n}} (\mu_{T_n}, \Phi_i^n(\mu_{T_i}^{(h)}))\ll e^{-a\sqrt[k+1]{\log T_n}}$. 

Finally, we have to estimate $\mT_P^{c_{T_n}} (\Phi_i^n(\mu_{T_i}^{(h)}), \rho_\infty(\cdot + c_{T_n}))$. To do this, it suffices to estimate the free energy $\mF(\Phi_i^n(\mu_{T_i}^{(h)}))$, as 
$$\mT_P\left((\Phi_i^n(\mu_{T_i}^{(h)}))^c,\rho_\infty\right) \le \const W_2^2\left((\Phi_i^n(\mu_{T_i}^{(h)}))^c,\rho_\infty \right) \le \const \mF(\Phi_i^n(\mu_{T_i}^{(h)})).$$
Indeed, $$\mF(\mu_{T_i}^{(h)}) = \mathcal{H}(\mu_{T_i}^{(h)}) + \iint \mu_{T_i}^{(h)}(\mathrm{d}x) W(x-y) \mu_{T_i}^{(h)}(\mathrm{d}y).$$
The first term here does not exceed $-\log vol(\mathcal{U}_h(0))\le d\cdot |\log (h/d)|$ (as the density of~$\mu_{T_i}^{(h)}$ does not exceed~$(h/d)^{-d}$), while the second term is bounded. Thus $\mF(\mu_{T_i}^{(h)}) \le C_6 \log n$ for some constant $C_6$. Hence, from the first part of Proposition~\ref{p:exp-decrease}, for $j =\left[ \left(\frac{C_6}{C_4} \log n\right)^{2/3} i\right]$, 
$$\mF(\Phi_i^j (\mu_{T_i}^{(h)})|\rho_\infty)\le C_4 + \frac{T_i}{T_j} \left(\mF(\mu_{T_i}^{(h)}|\rho_\infty) -C_4 \right) \le 2C_4.$$
Applying the second part, with $\Phi_i^j (\mu_{T_i}^{(h)})$ as a starting measure, we obtain 
$$\mF( \Phi_i^n (\mu_{T_i}^{(h)})) = \mF( \Phi_j^n\circ \Phi_i^j (\mu_{T_i}^{(h)})) \le C_5 e^{-a_1 \sqrt[k+1]{\log(T_n/T_j)}} \le e^{-a \sqrt[k+1]{\log T_n}}.$$
Thus, $\mF(\Phi_i^n (\mu_{T_i}^{(h)}))\le e^{-a \sqrt[k+1]{\log T_i}}$ and hence $$\mT_P\left((\Phi_i^n(\mu_{T_i}^{(h)}))^c,\rho_\infty\right) \le \const W_2^2\left((\Phi_i^n(\mu_{T_i}^{(h)}))^c,\rho_\infty \right) \le \const e^{-a \sqrt[k+1]{\log T_n}}.   \qedhere$$
\end{proof}

Let us now prove Propositions~\ref{p:one-step}--\ref{p:exp-decrease}.

\subsubsection{One-step error estimate}

This section is devoted to the proof of Proposition~\ref{p:one-step}.

To estimate the difference between the occupation measure of $X_t$ on $[T_n,T_{n+1}]$, and the measure $\Pi(\mu_t)$, we will first introduce another process, for which $\Pi(\mu_{T_n})$ is the stationary measure: the process with ``frozen'' measure~$\mu_{T_n}$. 
More precisely, on $[T_n,T_{n+1})$ we consider a process $Y$ with some choice of $Y_{T_n}$, satisfying
\begin{equation}\label{eq:Y-fix-mu}
\mathrm{d} Y_t = \sqrt{2}\, \mathrm{d}B_t - \nabla W*\mu_{T_n} (Y_t)\, \mathrm{d}t,
\end{equation} 
generated by \emph{the same} Brownian motion~$B_t$ as~$X_t$. In other words, the couple $(X_t,Y_t)$ satisfies
\begin{equation}\label{eq:XY}
 \left\{ 
   \begin{array}{l}
          \mathrm{d}X_t = \sqrt{2}\, \mathrm{d}B_t - \nabla W*\mu_{t} (X_t)\, \mathrm{d}t \\
	  \mathrm{d}Y_t = \sqrt{2}\, \mathrm{d}B_t - \nabla W*\mu_{T_n} (Y_t)\, \mathrm{d}t.
         \end{array}
 \right.
\end{equation}

The following lemma allows to control the difference between them:
\begin{lemma}\label{l:majX^1-X^2}
For all $t\in [T_n, T_{n+1}]$ we have
\begin{equation}\label{eq:distance-frozen}
|X_t - Y_t| \leq e^{-C_W (t-T_n)}|X_{T_n} - Y_{T_n}| + \frac{\Delta T_n}{T_n C_W} P(2L_n).
\end{equation}
\end{lemma}
\begin{proof}
The process $X_t -Y_t$ is of class $C^1$. We  
compute 
$$
\frac{\mathrm{d}}{\mathrm{d}t}(X_t -Y_t) = -( \nabla W*\mu_t(X_t) - \nabla W*\mu_{T_n}(Y_t)).
$$ 
Adding and substracting $\nabla W*\mu_{T_n}(X_t)$, we see 
$$
\mathrm{d}(X_t -Y_t) = -\left[\nabla W*(\mu_t-\mu_{T_n})(X_t) - 
(\nabla W*\mu_{T_n}(Y_t)-\nabla W*\mu_{T_n}(X_t)) \right] \mathrm{d}t.
$$ 
The last term can be rewritten as $$- (\nabla W*\mu_{T_n}(Y_t)-\nabla W*\mu_{T_n}(X_t)) = \frac{1}{T_n} \int_0^{T_n} \int_0^1 \nabla^2 W|_{uY_t + (1-u) X_t-X_s} \cdot (X_t -Y_t) \mathrm{d}u\, \mathrm{d}s.$$ Noting the first term as $D_t$, and putting a scalar product with $X_t-Y_t$, we see
\begin{eqnarray*}
\frac{1}{2} \frac{\mathrm{d}}{\mathrm{d}t}|X_t-Y_t|^2 &=& (D_t, X_t -Y_t)\\ 
&-& \frac{1}{T_n}\int_0^{T_n} \left(\int_0^1 \nabla^2 W|_{(uY_t + (1-u) X_t) -X_s^1} \mathrm{d}u\, \cdot (X_t -Y_t),X_t -Y_t \right)\, \mathrm{d}s\\
&\leq & (D_t, X_t-Y_t) - C_W |X_t - Y_t|^2.
\end{eqnarray*}
Thus, $\frac{\mathrm{d}}{\mathrm{d}t} |X_t - Y_t|^2 \leq -2C_W |X_t - Y_t|^2 + 2|D_t| |X_t - Y_t|$. Redividing by $2|X_t - Y_t|$, we obtain 
\begin{equation}\label{eq:D-diff}
\frac{\mathrm{d}}{\mathrm{d}t} |X_t - Y_t| \leq |D_t| - C_W |X_t - Y_t|.
\end{equation}
Finally, notice that $|D_t| \leq P(2L_n) \frac{\Delta T_n}{T_n}$, as it is the difference between the forces generated at $X_t$ by $\mu_{T_n}$ and by $\mu_t = \mu_{T_n} + \frac{t- T_n}{t} \left(\mu_{[T_n, t]} -\mu_{T_n} \right)$. Solving $\dot{u}_t = P(2L_n) \frac{\Delta T_n}{T_n} - C_W u_t$, we obtain the desired estimate for the difference $|X_t-Y_t|$ on the interval $[T_n, T_{n+1}]$.
\end{proof}

For what follows (see Proposition~\ref{p:CG} and Lemma~\ref{l:CG} below), we will have to assume that the initial distribution of $Y_{T_n}$ is absolutely continuous with respect to $\Pi(\mu_{T_n})$, and to use an estimate on its density. So finally, we define the process $Y_{t}$ for all $t$ in the following way: for every interval $[T_n,T_{n+1})$ the initial value $Y_{T_n}$ is chosen randomly with respect to the restriction of $\Pi(\mu_{T_n})$ on the unit ball $\mathcal{U}_1(c_{T_n})$. On each new interval, the choice is independent of $X$ and of all the past. Then, inside the interval $(T_n,T_{n+1})$, the couple $(X_n,Y_n)$ satisfies~\eqref{eq:XY}.

Let us compare the occupation measures of the processes $X$ and $Y$ on these intervals of time. Denote by $\tilde{\mu}_{[T_n,T_{n+1}]}$ the occupation measure of $Y$ on the interval $[T_n,T_{n+1}]$. Then, we have the following:

\begin{lemma}\label{l:tilde}
For any family of choices $Y_{T_n}\in \mathcal{U}_1(c_{T_n})$, we have 
$$
\mT_P^{c_{T_n}}(\mu_{[T_n, T_{n+1}]}, \tilde{\mu}_{[T_n, T_{n+1}]})= o(T_n^{-1/5}), \quad \text{ as } n\to\infty,
$$
provided that for $n$ sufficiently big $L_n\le C_3' \log n$.
\end{lemma}

\begin{proof}
The measures $\mu_{[T_n,T_{n+1}]}$ and $\tilde{\mu}_{[T_n, T_{n+1}]}$ are both images of the normalized Lebesgue measure $\frac{1}{\Delta T_n} \Leb_{[T_n, T_{n+1}]}$ under the maps $X_{\bullet}$ and $Y_{\bullet}$ respectively. So, consider the transport $\xi_s(t)=(1-s) X_t+ sY_t$ between them.

Using this transport, we have an estimate
\begin{eqnarray}\label{eq:transport-est}
\mT_P^{c_{T_n}} (\mu_{[T_n,T_{n+1}]}, \tilde{\mu}_{[T_n,T_{n+1}]} ) \le  \frac{1}{\Delta T_n} \int_{T_n}^{T_{n+1}}\int_0^1 P((1-s) X_t+ sY_t-c_{T_n}) |X_t-Y_t| \, \mathrm{d}s \,  \mathrm{d}t \notag\\
\le  \frac{1}{\Delta T_n} \int_{T_n}^{T_{n+1}} P(\max(|X_t-c_{T_n}|,|Y_t-c_{T_n}|)) |X_t-Y_t| \,  \mathrm{d}t.%\notag
\end{eqnarray}
By definition of $L_n$, we have $\forall t\in[T_n, T_{n+1}], \quad |X_t-c_{T_n}| \le L_n$ and due to Lemma~\ref{l:majX^1-X^2},
\begin{eqnarray*}
|Y_t - X_t| &\le & e^{-C_W (t-T_n)}|X_{T_n} - Y_{T_n}| + \frac{\Delta T_n}{T_n C_W} P(2L_n)\\ 
&\le & L_n + 1 + \frac{\Delta T_n}{T_n C_W} P(2 L_n) \le L_n + 2,
\end{eqnarray*}
provided that $L_n\le C_3'\log n$ and $n$ is sufficiently big. This implies that 
$$
|Y_t- c_{T_n}| \le |Y_t - X_t| + |X_t - c_{T_n}| \le 2 L_n +2.
$$

Now, substituting the obtained estimates to the right-hand side of~\eqref{eq:distance-frozen},
we see that 
\begin{multline*}
%&\mT_P^{c_{T_n}}& (\mu_{[T_n,T_{n+1}]}, \tilde{\mu}_{[T_n,T_{n+1}]} ) \le \\ 
%&\le &
% \frac{1}{\Delta T_n}   \int_{T_n}^{T_{n+1}} P(2L_n + 2) \cdot \left(e^{-C_W (t-T_n)} (L_n +1) + \frac{\Delta T_n}{T_n C_W} P(2L_n) \right) \mathrm{d}t \\ &\le & 
% \frac{P(2 L_n +2) P(2 L_n) \Delta T_n }{C_W T_n} + \frac{P(2L_n +2 ) (L_n+1)}{C_W \Delta T_n}= o(T_n^{-1/5})
\mT_P^{c_{T_n}} (\mu_{[T_n,T_{n+1}]}, \tilde{\mu}_{[T_n,T_{n+1}]} ) \le \\ 
\le  \frac{1}{\Delta T_n}   \int_{T_n}^{T_{n+1}} P(2L_n + 2) \cdot \left(e^{-C_W (t-T_n)} (L_n +1) + \frac{\Delta T_n}{T_n C_W} P(2L_n) \right) \mathrm{d}t\le \\ 
\le \frac{P(2 L_n +2) P(2 L_n) \Delta T_n }{C_W T_n} + \frac{P(2L_n +2 ) (L_n+1)}{C_W \Delta T_n}= o(T_n^{-1/5})
\end{multline*}
(we have used that $\Delta T_n\sim T_n^{1/3}$, and once again the logarithmic growth of $L_n$).
\end{proof}

Now, we will compare the occupation measure $\tilde{\mu}_{[T_n,T_{n+1}]}$ with $\Pi(\mu_{T_n})$. To do this, we use Proposition~1.2 of Cattiaux~\& Guillin~\cite{CG} (see also Wu \cite{Wu}), stating that the trajectory mean of a function $\psi$ is, with a probability close to 1 that can be exponentially controlled, close to its stationary mean. Namely, this proposition says the following:

\begin{proposition}[Cattiaux~\& Guillin~\cite{CG}]\label{p:CG} 
Given a process $\xi$ with a stationary measure $m$ and Poincar\'e constant $C_P$,  
an initial measure~$\nu$ and a function $\psi$ satisfying $|\psi|\le 1$, one has for any $0<\rho<1$ and $t>0$
$$
\mathbb{P}_{\nu} \left(\frac{1}{t} \int_0^t \psi(\xi_s) \, \mathrm{d}s - \int \psi \mathrm{d}m \ge \rho \right) \le \left\| \frac{\mathrm{d}\nu}{\mathrm{d}m} \right\|_{L_2(m)} 
\exp\left( - \frac{t \rho^2}{8 C_P Var_{m} (\psi)}\right).
$$
\end{proposition}

We will use this proposition with $\psi$ being the indicator function $\psi=\1_M$ of various sets $M$: it then allows to compare the occupation measure of the set $M$ to its $\Pi(\mu_{T_n})$-measure.

We know that $m=\Pi(\mu_{T_n})$ is the unique stationary measure of the drifted Brownian motion~\eqref{eq:Y-fix-mu}. Also, the Poincar\'e constant for this process is $2C_W$ (see~\cite{ane-etc}).

To proceed, we have to declare the initial measure $\nu=\nu_n$ for $Y_{T_n}$, and we choose it to be the measure $\Pi(\mu_{T_n})$ restricted to the ball $\mathcal{U}_1(c_{T_n})$ and then normalized accordingly. Then,
$$
\left\| \frac{\mathrm{d}\nu_n}{\mathrm{d}\Pi(\mu_{T_n})} \right\|_{L_2(\Pi(\mu_{T_n}))} = \frac{1}{\Pi(\mu_{T_n})(\mathcal{U}_1(c_{T_n}))} \le c_E= \const,
$$
the latter inequality is due to the exponential tails of~$\Pi(\mu_{T_n})$. 
Having made these choices, we are going to prove the following 

\begin{lemma}\label{l:CG}
As $n\to\infty$, we have almost surely
$$\mT_P^{c_{T_n}} (\tilde{\mu}_{[T_n, T_{n+1}]}, \Pi(\mu_{T_n})) = O((\Delta T_n)^{- \min \left(8C_W,\frac{1}{5d} \right) }).$$
\end{lemma}
\begin{proof}
The previous estimates imply that the process $Y_t$ on $[T_n,T_{n+1}]$ almost surely for all $n$ sufficiently big stays inside the ball 
$
\mathcal{U}_{R_n}(c_{T_n}),
$
where $R_n:=3 L_n$. Now, take this ball and cut it into some number $N_n$ parts $M_1,\dots, M_{N_n}$ of diameter less 
than $\varepsilon_n:=\frac{2dR_n}{\sqrt[d]{N_n}}$ (by cubic the grid with the step 
$2R_n/\sqrt[d]{N_n}$, that is decomposing each of the coordinate segments of length $2R_n$ into $\sqrt[d]{N_n}$ parts).
We will choose and fix the number $N_n$ later.

For each of these parts, choose 
$$
\rho_j:=\max \left(\frac{1}{N_n^2}, \frac{\Pi(\mu_{T_n})(M_j)}{N_n} \right).
$$
Let $\psi_j = \1_{M_j}$. Then, the probability that all the empirical measures $\tilde{\mu}_{[T_n,T_{n+1}]}(M_j)$ are $\rho_j$-close to their ``theoretical'' values~$\Pi(\mu_{T_n})(M_j)$ is at least 
$$
1-2c_E \sum_{j=1}^{N_n} \exp \left( - \frac{\rho_j^2 \Delta T_n}{16 C_W Var_{\Pi(\mu_{T_n})} (\psi_j)}\right).
$$
As the variance $Var_{\Pi(\mu_{T_n})} (\psi_j)$  
does not exceed $\Pi(\mu_{T_n})(\psi_j)$, we have a lower bound for the probability by
$$
1-2c_E \sum_{j=1}^{N_n} \exp \left( - \frac{\rho_j \Delta T_n}{16C_W} \cdot \frac{\rho_j}{\Pi(\mu_{T_n}) (\psi_j)}\right) \ge 1-2 N_n c_E \exp \left(- \frac{\Delta T_n}{16 C_W N_n^3} \right),
$$
as $\frac{\rho_j}{\Pi(\mu_{T_n}) (\psi_j)} \ge \frac{1}{N_n}$ and $\rho_j\ge \frac{1}{N_n^2}$.

So, taking $N_n=\sqrt[10]{T_n}\sim (\Delta T_n)^{3/10}$, we see that the series 
$$
\sum_n  N_n \exp \left(- \frac{\Delta T_n}{16C_WN_n^3} \right) \asymp \sum_n (\Delta T_n)^{3/10} \exp \left(- (\Delta T_n)^{1/10} \right)
$$ 
converges, so almost surely for all $n$ sufficiently big, all the closeness conditions on the occupation measures are satisfied: the measures $\tilde{\mu}_{[T_n,T_{n+1}]}(M_j)$ are a.s. $\rho_j$-close to $\Pi(\mu_{T_n})(M_j)$.

Now, let us estimate the $c_{T_n}$-centered distance $\mT_P^{c_{T_n}}(\tilde{\mu}_{[T_n,T_{n+1}]}, \Pi(\mu_{T_n}))$, provided that these conditions are fulfilled. Indeed, first transport inside each $M_j$ the part $\min(\tilde{\mu}_{[T_n,T_{n+1}]}, \Pi(\mu_{T_n}))$: we pay at most $P(3 L_n) \varepsilon_n = O\left((\Delta T_n)^{-\frac{1}{5d}}\right)$. 
Next, bring the exterior part of $\Pi(\mu_{T_n})$ to the ball $\mathcal{U}_{R_n}(c_{T_n})$: due to the exponential decrease estimates, we pay at most $$\int_{R_n}^\infty P(r) \mathrm{d}(1-Ce^{-C_W r}) \sim R_n^{k+1} e^{-C_W R_n} = O\left((\Delta T_n)^{-8C_W}\right)$$ as $R_n = 3 \log T_n$. Finally, let us re-distribute the parts left: we pay at most 
\begin{eqnarray*}
\sum_{j=1}^{N_n} \rho_j R_n P(R_n) &=& R_n P(R_n) \sum_{j=1}^{N_n} \max{\left( \frac{1}{N_n^2}, \frac{\Pi(\mu_{T_n})(M_j)}{N_n} \right)}\\ 
& \le & R_n P(R_n) \sum_{j=1}^{N_n} \left( \frac{1}{N_n^2} + \frac{\Pi(\mu_{T_n})(M_j)}{N_n} \right)\\
&\le & 2R_n P(R_n) \frac{1}{N_n} = O\left( (\Delta T_n)^{-1/5} \right).
\end{eqnarray*}
Adding these three estimates, we obtain the desired $\mathcal{T}_P^{c_{T_n}}(\tilde{\mu}_{[T_n,T_{n+1}]}, \Pi(\mu_{T_n})) = O\left( (\Delta T_n)^{-\beta} \right)$ with $\beta = \min (8C_W, (5d)^{-1})$.
\end{proof}

Putting Lemmas~\ref{l:tilde} and~\ref{l:CG} together, and recalling that $\Delta T_n \sim T_n^{\frac{1}{3}}$, we conclude that almost surely, for all $n$ sufficiently big,
$$
\mT_P^{c_{T_n}} (\mu_{[T_n, T_{n+1}]}, \Pi(\mu_{T_n})) \le T_n^{- \min \left(\frac{8}{3}C_W,\frac{1}{15d} \right)}.
$$
Proposition~\ref{p:one-step} is proven.

%\newpage

\subsubsection{Euler method error control}\label{s:Euler}

\begin{proof}[Proof of Proposition~\ref{p:Euler-smooth}]

We prove the proposition by induction on~$j$. The case $j=i$ is obvious: the only term in the right-hand side is $C_1 h$, being an estimate for the distance to the smoothened convolution:
\begin{eqnarray*}
\mT_P^{c_{T_i}} ( \mu_{T_i}^{(h)}, \mu_{T_i} ) &\le & \int_{\Rn} \int_{\mathcal{U}_h(0)} |v| \cdot P\left((\max (|x-c_{T_i}|,|x+v-c_{T_i}|)\right) \, \frac{\mathrm{d}v}{vol(\mathcal{U}_{h}(0))} \, \mathrm{d}\mu(x) \\ 
&\le & \int_{\Rn} P(|x-c_{T_i}|+h) \cdot h \, \mathrm{d}\mu(x) \le P(h) \|\mu(\cdot + c_{T_i})\|_P \cdot h = C_1 h,
\end{eqnarray*}
provided that $h\le 1$ (because the norm $\|\mu^c_{T_i}\|_P$ is bounded due to the exponential tails of $\mu^c$).

Let us now check the step of induction. Namely, assume that the conclusion holds for some $j\ge i$, and check it for $j+1$. To do this, first shift the center of the translation distance from $c_{T_{j+1}}$ to $c_{T_j}$: from Proposition~\ref{p:expo-decrease}
$$
\mT_P^{c_{T_{j+1}}} (\cdot,\cdot) \le (1+\const |c_{T_{j+1}}-c_{T_j}|) \mT_P^{c_{T_j}} (\cdot,\cdot),
$$
provided that $|c_{T_{j+1}}-c_{T_j}|\le 1$. On the other hand, we have by Lemma~\ref{l:mT_P}
$$
|c_{T_{j+1}}-c_{T_j}|\le \Lip_{K_{\alpha,C}} (c) \cdot \mT_P^{c_{T_j}}(\mu_{T_{j+1}},\mu_{T_j}) \le \const \frac{\Delta T_j}{T_{j+1}},
$$
so finally 
\begin{equation}\label{eq:Euler-first}
\mT_P^{c_{T_{j+1}}} (\cdot,\cdot) \le \left(1+\const \cdot \frac{\Delta T_j}{T_{j+1}}\right) \mT_P^{c_{T_j}} (\cdot,\cdot) \le \left( \frac{T_{j+1}}{T_j} \right)^{A_1} \mT_P^{c_{T_j}} (\cdot,\cdot).
\end{equation}

Now, the map $\Pi$ is Lipschitz on $K_{\alpha,C}$ by Proposition~\ref{prop:exp-decrease}, so for any two measures $\nu_1,\nu_2$ one has 
$$
\mT_P(\Phi_{j}^{j+1} (\nu_1), \Phi_{j}^{j+1} (\nu_2)) \le \left(1+ \frac{\Delta T_j}{T_{j+1}} (\Lip_{K_{\alpha,C}}(\Pi)+1) \right) \mT_P(\nu_1,\nu_2) \le \left( \frac{T_{j+1}}{T_j} \right)^{A_2} \mT_P (\nu_1,\nu_2).
$$

Substituting for $\nu_1$ and $\nu_2$ respectively the translated by $c_{T_j}$ images of measures $\Phi_i^j(\mu_{T_i}^{(h)})$ and $\mu_{T_j}$ respectively, we see that
\begin{equation}\label{eq:Euler-next}
\mT_P^{c_{T_j}} (\Phi_i^{j+1}(\mu_{T_i}^{(h)}),\Phi_j^{j+1}(\mu_{T_j})) \le \left(\frac{T_{j+1}}{T_j} \right)^{A_2} \mT_P^{c_{T_j}} (\Phi_i^{j}(\mu_{T_i}^{(h)}),\mu_{T_j}).
\end{equation}
Now, using that by Proposition~\ref{p:one-step}, 
$$
\mT_P^{c_{T_j}} (\Phi_j^{j+1}(\mu_{T_j}),\mu_{T_{j+1}}) \le \left( \frac{\Delta T_{j}}{T_{j+1}} \right) (\Delta T_j)^{-\beta},
$$
with $\beta = \min (8C_W, (5d)^{-1})$, we see that
\begin{multline}\label{eq:Euler-triangle}
\mT_P^{c_{T_{j+1}}} (\Phi_i^{j+1}(\mu_{T_i}^{(h)}),\mu_{T_{j+1}}) \le \left( \frac{T_{j+1}}{T_j} \right)^{A_1} \mT_P^{c_{T_j}} (\Phi_i^{j+1}(\mu_{T_i}^{(h)}),\mu_{T_{j+1}}) \le 
\\ \le 
\left( \frac{T_{j+1}}{T_j} \right)^{A_1} \left (\mT_P^{c_{T_j}} (\Phi_i^{j+1}(\mu_{T_i}^{(h)}),\Phi_j^{j+1}(\mu_{T_j})) +\mT_P^{c_{T_j}} (\Phi_j^{j+1}(\mu_{T_j}),\mu_{T_{j+1}}) \right) \le 
\\ \le
\left( \frac{T_{j+1}}{T_j} \right)^{A_1} \left (\left( \frac{T_{j+1}}{T_j} \right)^{A_2} \mT_P^{c_{T_j}} (\Phi_i^{j}(\mu_{T_i}^{(h)}),\mu_{T_j}) +\frac{\Delta T_j}{T_{j+1}} (\Delta T_j)^{-\beta} \right)  = 
\\ = 
\left( \frac{T_{j+1}}{T_j} \right)^{A_1+A_2} \mT_P^{c_{T_j}} (\Phi_i^{j}(\mu_{T_i}^{(h)}),\mu_{T_j}) + \left( \frac{T_{j+1}}{T_j} \right)^{A_1} \frac{\Delta T_j}{T_{j+1}} (\Delta T_j)^{-\beta}. 
\end{multline}

Finally, we fix the choice of $A:=A_1+A_2$, and, using the induction assumption, the right-hand side of~\eqref{eq:Euler-triangle} is not greater than
\begin{multline*}
\left( \frac{T_{j+1}}{T_j} \right)^{A}\left(  
\sum_{k=i}^{j-1} \frac{\Delta T_k}{T_{k+1}} (\Delta T_{k})^{-\beta} \left(\frac{T_j}{T_k}\right)^A 
+ C_1 h \left(\frac{T_j}{T_i}\right)^A \right)  + \left( \frac{T_{j+1}}{T_j} \right)^{A_1} \frac{\Delta T_j}{T_{j+1}} (\Delta T_j)^{-\beta} \le 
\\ \le
\sum_{k=i}^{j-1} \frac{\Delta T_k}{T_{k+1}} (\Delta T_{k})^{-\beta} \left(\frac{T_{j+1}}{T_k}\right)^A 
+ C_1 h \left(\frac{T_{j+1}}{T_i}\right)^A + \left( \frac{T_{j+1}}{T_j} \right)^A \frac{\Delta T_j}{T_{j+1}} (\Delta T_j)^{-\beta} = 
\\ =
\sum_{k=i}^j \frac{\Delta T_k}{T_{k+1}} (\Delta T_{k})^{-\beta} \left(\frac{T_{j+1}}{T_k}\right)^A 
+ C_1 h \left(\frac{T_{j+1}}{T_i}\right)^A.
\end{multline*}
The induction step is proved.

\end{proof}

%\newpage

\subsubsection{Decrease of energy}\label{sss:dec-energy}

This section is devoted to the proof of Proposition~\ref{p:exp-decrease}. To estimate the decrease of energy, we will need the following
\begin{lemma}\label{l:maj-varphi}
For any $\mu\in K_{\alpha,C}$, we have $\varphi_\mu(\mu) -\varphi_\mu(\Pi(\mu)) \ge g(\mF(\mu|\rho_\infty)),$ where $$g(E) = \left\{ \begin{array}{l}
	C_7 \frac{E}{|\log E|^k}, \quad 0\le E \le \varepsilon_0<1\\
	\frac{E}{\varepsilon_0} g(\varepsilon_0), \quad \varepsilon_0< E\le \varepsilon_1 \\
	E + (g(\varepsilon_1) -\varepsilon_1), \quad E>\varepsilon_1
        \end{array} 
\right.
$$ is an increasing continuous function, and the constants $C_7,\varepsilon_0,\varepsilon_1$ depend only on $\alpha$ and $C$.
\end{lemma}

We postpone its proof, but we use it as a motivation for the next result, which immediately implies Proposition~\ref{p:exp-decrease}:
\begin{lemma}\label{l:proof-p9}
There exists $n_0$ such that for any $\mu\in K_{\alpha,C}$ and for any $j\ge i\ge n_0$: $$\mF(\Phi_i^j (\mu)|\rho_\infty) \le y(T_j),$$ where $y$ is the solution to 
\begin{equation}\label{eq:ydot}
\dot{y} = -\frac{1}{t} \frac{g(y)}{2}, 
\end{equation}
with the initial condition $y(T_i) = \max (\mF(\mu_{T_i}),1)$.
\end{lemma}
Proposition~\ref{p:exp-decrease} is its immediate corollary, as the solution of~\eqref{eq:ydot} decreases exponentially for big energies~$y$ and has the form $y(t) = \exp \{-\sqrt[k+1]{\frac{C_7}{2}(k+1) \log (t/T_0)}\}$ for $y\le \varepsilon_0$ (what happens for $t$ large enough). 

We will need the following corollary to Lemma~\ref{l:equality-varphi}:
\begin{corollary}\label{cor:varphi}
For any fixed $\alpha,C$, there exists $C''$ such that for all $\mu\in K_{\alpha,C}$, for all $0<\lambda<1$, $$\mF((1-\lambda)\mu + \lambda \Pi(\mu)|\rho_\infty) \le \mF(\mu|\rho_\infty) -\lambda \left(\varphi_\mu(\mu) -\varphi_\mu(\Pi(\mu)) \right) + C'' \lambda^2.$$
\end{corollary}
\begin{proof}
For $\mu\in K_{\alpha,C}$, the integral that is the coefficient before $\lambda^2$ is uniformly bounded.
\end{proof}

Let us now prove the previous lemmas.
\begin{proof}[Proof of Lemma~\ref{l:proof-p9}]
Recall that, due to Corollary~\ref{cor:varphi}, we have once $\mu\in K_{\alpha,C}$, $$\mF((1-\lambda)\mu + \lambda \Pi(\mu)|\rho_\infty) \le \mF(\mu|\rho_\infty) -\lambda \left(\varphi_\mu(\mu) -\varphi_\mu(\Pi(\mu)) \right) + C'' \lambda^2.$$ Now note that, if $n_0$ is chosen sufficiently big, we have for any $j$:
\begin{equation}\label{eq:g-y}
C'' \frac{\Delta T_j}{T_{j+1}} \le \frac{g(y(T_j))}{3}.
\end{equation}
Indeed, the left-hand side of \eqref{eq:g-y} decreases as $\frac{1}{j}$, while its right-hand side decreases as $\exp \{-\sqrt[k+1]{\frac{C_7}{2} (k+1) \log T_j}\} \gg \frac{1}{j}$. Now, for every $\check{\mu}_j := \Phi_i^j(\mu)$, we have $\check{\mu}_j \in K_{\alpha,C}$ due to Lemma~\ref{l:K-alpha-C} and hence due to Lemma~\ref{l:maj-varphi}: $$\varphi_{\check{\mu}_j}(\check{\mu}_j) - \varphi_{\check{\mu}_j}(\Pi(\check{\mu}_j)) \ge g(\mF(\check{\mu}_j|\rho_\infty)).$$ Hence, proving the statement of the lemma by induction on $j$, we have to deduce from $\mF(\check{\mu}_j|\rho_\infty) \le y(T_j)$ the analogous statement for $\check{\mu}_{j+1}$, given that $$\mF(\check{\mu}_{j+1}|\rho_\infty) \le y(T_j) -g(y(T_j)) \frac{\Delta T_j}{T_{j+1}} + C'' \left(\frac{\Delta T_j}{T_{j+1}} \right)^2 \le y(T_j) - \frac{2}{3}g(y(T_j))\frac{\Delta T_j}{T_{j+1}}.$$
Let $\theta_j = \log T_j$. Then, $\Delta \theta_j := \theta_{j+1}-\theta_j \le \frac{4}{3}\frac{\Delta T_j}{T_{j+1}}$ for all $j$ large enough. So, once again asking $n_0$ to be chosen sufficiently big, we have $$\mF(\check{\mu}_{j+1}|\rho_\infty) \le y(T_j) - \frac{2}{3}\cdot \frac{3}{4}g(y(T_j)) \Delta \theta_j = \tilde{y}(\theta_j) - \frac{g(\tilde{y}(\theta_j))}{2}\Delta \theta_j,$$ where $\tilde{y}(\theta)=y(e^\theta)$. We conclude by noticing that $g(y)$ is an increasing function of $y$. So, as $\tilde{y}(\theta)$ is solution to the equation $\tilde{y}(\theta)' = -\frac{g(\tilde{y}(\theta))}{2}$, we have $$\tilde{y}(\theta_j) -\frac{g(\tilde{y}(\theta_j))}{2} \Delta \theta_j \le \tilde{y}(\theta_{j+1}),$$ hence $\mF(\check{\mu}_{j+1}|\rho_\infty) \le \tilde{y}(\theta_{j+1}) = y(T_{j+1})$, thus proving the induction step.
\end{proof}

\begin{proof}[Proof of Lemma~\ref{l:maj-varphi}]
Note first that, for $\mu\in K_{\alpha,C}$, the integral $\iint (\mu-\Pi(\mu))(\mathrm{d}x) W(x-y) (\mu-\Pi(\mu))(\mathrm{d}y)$ is bounded by a uniform constant. Thus, due to Lemma~\ref{l:equality-varphi}, $\varphi_\mu(\mu)-\varphi_\mu(\Pi(\mu))$ admits a lower bound
\begin{equation}\label{eq:varphi-free}
 \varphi_\mu(\mu) -\varphi_\mu(\Pi(\mu)) \ge \mF(\mu|\rho_\infty) -C_\Delta
\end{equation}
with the constant $C_\Delta$ being uniform over all $\mu\in K_{\alpha,C}$.

Now, let us give another way to estimate the difference $\varphi_\mu(\mu) -\varphi_\mu(\Pi(\mu))$. Indeed, $\Pi(\mu)$ is the global minimiser of $\mF$, hence for any measure $\rho$, we have
\begin{equation}\label{eq:varphi-min}
\varphi_\mu(\mu) -\varphi_\mu(\Pi(\mu)) \ge \varphi_\mu(\mu) -\varphi_\mu(\rho).
\end{equation}
Recall that the free energy functional $\mF$ is displacement convex. Denote by $\xi_s = (1-s)\xi_0+s\xi_1$, $0\le s\le 1$, the quadratic Wasserstein optimal transport between $\mu=\{$law of~$\xi_0\}$ and $\rho_\infty(\cdot +\mu)=\{$law of~$\xi_1\}$ and let $\nu_s=\{$law of~$\xi_s\}$. Then, $$\mF(\nu_s|\rho_\infty) \ge (1-s) \mF(\mu|\rho_\infty).$$
Thus, we have due to Lemma~\ref{l:equality-varphi},
\begin{eqnarray*}
\varphi_\mu(\mu) -\varphi_\mu(\nu_s) = \mF(\mu|\rho_\infty) -\mF(\nu_s|\rho_\infty) + \frac{1}{2}\iint (\nu_s-\mu)(\mathrm{d}x) W(x-y)(\nu_s-\mu)(\mathrm{d}y)\\
\ge s\mF(\mu|\rho_\infty)+ \frac{1}{2}\iint (\nu_s-\mu)(\mathrm{d}x) W(x-y)(\nu_s-\mu)(\mathrm{d}y).
\end{eqnarray*}
Let us now estimate the second term in the right-hand side of this inequality. Indeed, let $(\eta_0,\eta_1)$ be an independent copy of $(\xi_0,\xi_1)$. Then 
\begin{eqnarray*}
\iint W(x-y) (\nu_s-\mu)(\mathrm{d}x) (\nu_s-\mu)(\mathrm{d}y)=\\ 
= \mathbb{E} \left[W(\xi_0 - \eta_0) - W(\xi_s -\eta_0) - W(\xi_0-\eta_s) + W(\xi_s-\eta_s) \right].
\end{eqnarray*} 
For any fixed $L$, we can divide this expectation into two parts: the one corresponding to $\underset{i,j\in \{0,1\}}{\max} (|\xi_i|,|\eta_j|) >L$ and the one with $|\xi_i|,|\eta_i|\leq L$ for $i=0,1$. We also remind that $\nu_i \in K_{\alpha, C_2}$ for $i=0,1$ and that $P$ controls $W$ as well as its first and second derivatives. So, there exists a positive constant $\tilde{C}$ such that 
\begin{eqnarray*}
& & \left|\mathbb{E} \left[W(\xi_0 - \eta_0) - W(\xi_s -\eta_0) - W(\xi_0-\eta_s) + W(\xi_s-\eta_s) \right] \right|\\
&\leq & \left|\mathbb{E} \left[W(\xi_0 - \eta_0) - W(\xi_s -\eta_0) - W(\xi_0-\eta_s) + W(\xi_s-\eta_s) \right]\1_{\left\{\underset{i,j\in \{0,1\}}{\max} (|\xi_i|,|\eta_j|) \leq L\right\}} \right|\\ 
&+& \int_L^\infty W(2l) \, \mathrm{d} F_{\max(\xi_0,\xi_1,\eta_0,\eta_1)}(l) \\
&\leq & \mathbb{E} \left[ \max_{|x|\leq 4L} P(|x|) |\xi_0 -\xi_s| |\eta_0-\eta_s|\1_{\left\{\underset{i,j\in \{0,1\}}{\max} (|\xi_i|,|\eta_j|) \leq L\right\} } \right] + 4\int_L^\infty P(2l) \, \mathrm{d}(1-C_2 e^{-\alpha l})^4 \\
&\leq & s^2 P(4L) W_2^2 (\nu_0,\nu_1)+ \tilde{C}P(2L) e^{-\alpha L}.
\end{eqnarray*}
So, using the already mentioned comparison $W_2^2(\mu, \rho_\infty) \leq \frac{2}{C_W} \mF(\mu|\rho_\infty)$, we have 
\begin{eqnarray*}
\varphi_\mu(\mu) - \varphi_\mu(\nu_s) &\geq & s \mF(\mu|\rho_\infty) - s^2 P(4L) W_2^2 (\mu,\rho_\infty) - \tilde{C}P(2L) e^{-\alpha L}\\
&\geq & s \mF(\mu|\rho_\infty) - \frac{2}{C_W} s^2 P(4L) \mF(\mu|\rho_\infty) - \tilde{C}P(2L) e^{-\alpha L}.
\end{eqnarray*}
We decide from now on to fix $s= \frac{C_W}{4P(4L)}$, with the choice of $L$ to be fixed later. Then, $s -\frac{2}{C_W}s^2 P(4L) = s/2$ and
\begin{equation}\label{eq:4*}
\varphi_\mu(\mu) - \varphi_\mu(\nu_s) \ge \frac{C_W}{8P(4L)} \mF(\mu|\rho_\infty) - \tilde{C}P(2L) e^{-\alpha L}.
\end{equation}
For $\mF(\mu|\rho_\infty)$ sufficiently small, fixing $L= \frac{2}{\alpha} |\log \mF(\mu|\rho_\infty)|$, we have
\begin{equation}\label{eq:5*}
\frac{C_W}{16 P(4L)}\mF(\mu|\rho_\infty) \ge \tilde{C}P(2L) e^{-\alpha L}
\end{equation}
and hence the right-hand side of \eqref{eq:4*} is estimated from below by
$$\frac{C_W}{16 P(4L)}\mF(\mu|\rho_\infty) \ge \frac{C_7}{|\log \mF(\mu|\rho_\infty)|^k} \mF(\mu|\rho_\infty).$$
So taking $g(E) := \frac{C_7}{|\log E|^k} E$ for such values of $E=\mF(\mu|\rho_\infty)$, we have for such $E$'s the conclusion of lemma satisfied. Next, fixing $\varepsilon_0$ to be such that~\eqref{eq:5*} is satisfied for $\mF(\mu|\rho_\infty)\le \varepsilon_0$, and for any $\mF(\mu|\rho_\infty)\ge \varepsilon_0$, choosing the same $L$ as for $\mF(\mu|\rho_\infty)= \varepsilon_0$, we have $$\frac{C_W}{8P(4L)} \mF(\mu|\rho_\infty) - \tilde{C}P(2L) e^{-\alpha L} \ge \frac{\mF(\mu|\rho_\infty)}{\varepsilon_0} g(\varepsilon_0),$$ what allows to deduce
\begin{equation}\label{eq:6*}
\varphi_\mu(\mu) -\varphi_\mu(\Pi(\mu)) \ge \left\{\begin{array}{ll}
					g(\mF(\mu|\rho_\infty))\quad \text{if} \quad \mF(\mu|\rho_\infty)\le \varepsilon_0\\
					\frac{\mF(\mu|\rho_\infty)}{\varepsilon_0} g(\varepsilon_0) \quad \text{if} \quad \mF(\mu|\rho_\infty)> \varepsilon_0
                                            \end{array}
\right.
\end{equation}
Finally, taking the maximum between the right-hand side of~\eqref{eq:varphi-free} and~\eqref{eq:6*}, we obtain the desired conclusion.
\end{proof}

\subsection{Proof of Theorem \ref{t:centers}}\label{s:center}
As it has been already shown in~\eqref{eq:center-speed}, we have $$|c_{T_{n+1}} - c_{T_n}|\le \int_{T_n}^{T_{n+1}} \frac{P(|X_t-c_t|)}{C_W t} \mathrm{d}t \le \int_{T_n}^{T_{n+1}} \frac{P(L_{n} +|c_t - c_{T_n}|)}{C_W t} \mathrm{d}t \le P(L_{n}+C_3) \frac{\Delta T_n}{T_n}.$$
Thus, almost surely one has $osc_{t\in [T_n, T_{n+1}]} c_t \to 0$ as $n\to \infty$. So, to prove Theorem~\ref{t:centers}, it suffices to show that the sequence~$c_{T_n}$ converges almost surely.

Now, let us estimate the distance $c_{T_{n+1}}-c_{T_n}$. Indeed, 
$$
\mu_{T_{n+1}} = \mu_{T_n} + \frac{\Delta T_n}{T_{n+1}} (\mu_{[T_n,T_{n+1}]} - \mu_{T_n}).
$$
Translating $c_{T_n}$ to the origin, using the decrease estimates of \S\ref{s:discretization} and recalling that $c(\cdot):K_{\alpha,C}^0\to\Rn$ is a $\mT_P$-Lipschitz function, we see that
$$
|c_{T_{n+1}}-c_{T_n}|\le \Lip_{K_{\alpha,C}^0}(c)\cdot \frac{\Delta T_n}{T_{n+1}} \cdot \mT_P^{c_{T_n}} (\mu_{[T_n,T_{n+1}]}, \mu_{T_n}).
$$
As in~\S\ref{s:Euler}, the distance in the right-hand side can be estimated as a sum of two distances:
\begin{equation}\label{eq:triangle}
\mT_P^{c_{T_n}} (\mu_{[T_n,T_{n+1}]}, \mu_{T_n})\le \mT_P^{c_{T_n}} (\mu_{[T_n,T_{n+1}]}, \Pi(\mu_{T_n})) + \mT_P^{c_{T_n}} (\Pi(\mu_{T_n}), \mu_{T_n}).
\end{equation}
We already have an estimate for the first term in this sum:
$$
\mT_P^{c_{T_n}} (\mu_{[T_n, T_{n+1}]}, \Pi(\mu)) \le T_n^{- \min \left(\frac{8}{3}C_W,\frac{1}{15d} \right)}.
$$
On the other hand, the limit density $\rho_{\infty}$ is a fixed point of the map $\Pi$. And the map $\Pi$ being Lipschitz on $K_{\alpha,C}^0$, the second summand in~\eqref{eq:triangle} can be estimated as
$$
\mT_P^{c_{T_n}} (\Pi(\mu_{T_n}), \mu_{T_n}) \le (\Lip_{K_{\alpha,C}^0}(\Pi)+1) \cdot \mT_P (\mu_{T_n}^c, \rho_{\infty}).
$$
The latter distance is already estimated in the proof of Theorem~\ref{t:centered}: almost surely for $n$ sufficiently big, we have 
$$
\mT_P (\mu_{T_n}^c, \rho_{\infty}) \le \exp \{-a \sqrt[k+1]{\log T_n}\}.
$$
Finally, adding the estimates for the first and the second terms in~\eqref{eq:triangle}, we obtain that for all $n$ sufficiently big,
$$
\mT_P^{c_{T_n}} (\mu_{[T_n,T_{n+1}]}, \mu_{T_n}) \le T_n^{- \min \left(\frac{8}{3}C_W,\frac{1}{15d} \right)} + (\Lip_{K_{\alpha,C}^0}(\Pi)+1) \exp \{-a \sqrt[k+1]{\log T_n}\}
$$
and hence 
$$
|c_{T_{n+1}}-c_{T_n}|\le  \Lip_{K_{\alpha,C}^0}(c)\cdot \frac{\Delta T_n}{T_{n+1}} \left(T_n^{- \min \left(\frac{8}{3}C_W,\frac{1}{15d} \right)} + (\Lip_{K_{\alpha,C}^0}(\Pi)+1) \exp \{-a \sqrt[k+1]{\log T_n}\} \right).
$$
We choose $T_n = n^{3/2}$ and so $\frac{\Delta T_n}{T_n} \asymp n^{-1}$. Hence 
$$\sum_n |c_{T_{n+1}}-c_{T_n}| \le \const \sum_n \frac{1}{n^{1+ \min (4C_W, 1/(10d))}} + \const \sum_n \frac{\exp \{-a\sqrt[k+1]{\frac{3}{2}\log n}\}}{n}.
$$  
Both series in the right-hand side converge, and thus the series $\sum_n |c_{T_{n+1}}-c_{T_n}|$ converges almost surely. This concludes the proof.

\section*{Appendix 1: Singularity at $t=0$}\label{ss:x-mu-0}

%The last statement that we are going to prove in this paragraph is
Let us now prove that a solution to the equation~\eqref{eq:AABM1} with any initial
condition at $t=0$ (where the equation has a singularity) exists
and is unique. 

\begin{proposition}\label{l:existence-0}
For any $x_0$ and almost every trajectory~$B_t$ of the Brownian
motion, a (continuous at $t=0$) solution $X_t$ to the
equation~\eqref{eq:AABM1} with the initial condition $X_0=x_0$
exists on all the interval~$[0,+\infty)$ and is unique.
\end{proposition}
\begin{proof}
As Proposition~\ref{l:MBAA-global} provides us global existence and
uniqueness of solutions, starting from any arbitrary positive time
$r>0$, it suffices to check the existence and uniqueness on some
interval~$[0,\delta)$. For the sake of simplicity of notation, suppose that $x_0=0$.

Let $\delta_1>0$ be such that for all $0\le t\le \delta_1$, $\vert B_t \vert \le \frac{1}{2}$ and $\delta_1 \sup_{\vert x\vert \le 2} \vert \nabla W(x)\vert \le \frac{1}{3}$. We work on the trajectories, which are staying inside $\mathcal{U}_1(0)$, the unit ball centered in $x_0=0$. So, we consider $X_\bullet : [0,\delta_1)\rightarrow \mathcal{U}_1(0)$, $t\mapsto X_t$. Denote by $\mu_s^X$ the empirical measure of the process $X$. Then, the application $\chi: X\mapsto \tilde{X}$ is such that 
$$
\tilde{X}_t = B_t +\int_0^t \nabla W*\mu_s^{X}(\tilde{X}_s) \mathrm{d}s,
$$ 
is well-defined on this space, and $\tilde{X}_t$ also remains stuck in~$\mathcal{U}_1(0)$. Indeed, for any time $t\le \delta_1$, such that the solution~$\tilde{X}$ is defined on~$[0,t]$ and stays in~$\mathcal{U}_1(0)$, we have 
\begin{eqnarray*}
\left\vert \int_0^t \nabla W*\mu_s^{X}(\tilde{X}_s) \, \mathrm{d}s \right\vert &=&  \left\vert \int_0^t \frac{1}{s} \int_0^s  \nabla W(\tilde{X}_s-X_u)  \, \mathrm{d}u \, \mathrm{d}s \right\vert \\
&\le& \int_0^t \frac{1}{s} \int_0^s \sup_{\vert x \vert \le 2} \vert \nabla W(x) \vert \, \mathrm{d}u \, \mathrm{d}s \le \delta_1\sup_{\vert x \vert \le 2} \vert \nabla W(x) \vert \le \frac{1}{3}.
\end{eqnarray*}
Thus, if there existed a time $t_0 \le \delta_1$ such that $\vert \tilde{X}_{t_0} \vert \ge 7/8$ for the first time, then we would see that $\vert \tilde{X}_{t_0} \vert \leq 1/2 
 + 1/3$, which would contradict the bound $\vert \tilde{X}_{t_0} \vert \ge 7/8$. 
So, $\tilde{X}$ stays in $\mathcal{U}_1(0)$ for any $0\le t\le \delta_1$. 

Let us now show that for $\delta<\delta_1$ sufficiently small, the map~$\chi$ is a contraction on the space of continuous maps $X_{\bullet}$ from~$[0,\delta]$ to~$\mathcal{U}_1(0)$ with $X_0=0$. Indeed, consider now two trajectories $X^{(1)}$ and $X^{(2)}$, realizing a coupling with the same Brownian motion, and their respective images (by $\chi$) $\tilde{X}^{(1)}$ and $\tilde{X}^{(2)}$. Then, denoting by $\Lip(W)$ the Lipschitz constant of $\nabla W$ on the 2-radius ball, $\Lip(W):= \sup \{||\nabla^2 W(x)||: \vert x \vert \le 2\}$, we have
\begin{eqnarray*}
\vert \tilde{X}^{(1)}_t - \tilde{X}^{(2)}_t \vert &=& 
\left\vert \int_0^t \nabla W*\mu_s^{X^{(1)}}(\tilde{X}_s^{(1)}) \mathrm{d}s -  \int_0^t \nabla W*\mu_s^{X^{(2)}}(\tilde{X}_s^{(2)}) \mathrm{d}s \right\vert \\ 
&=& \left\vert \int_0^t \frac{1}{s} \int_0^s  \nabla W(\tilde{X}_s^{(1)}-X_u^{(1)}) -  \nabla W(\tilde{X}_s^{(2)}-X_u^{(2)}) \, \mathrm{d}u \, \mathrm{d}s \right\vert \\
&\le & \int_0^t \frac{1}{s} \int_0^s  \vert  \nabla W(\tilde{X}_s^{(1)}-X_u^{(1)}) -  \nabla W(\tilde{X}_s^{(2)}-X_u^{(2)})  \vert\, \mathrm{d}u \, \mathrm{d}s \\
&\le & \int_0^t \frac{1}{s} \int_0^s  \Lip(W) (\vert  \tilde{X}_s^{(1)} - \tilde{X}_s^{(2)} \vert + \vert  X_u^{(1)} - X_u^{(2)}  \vert )\, \mathrm{d}u \, \mathrm{d}s \\
&\le& t \Lip(W) (\vert \vert  \tilde{X}^{(1)} - \tilde{X}^{(2)} \vert\vert_{C([0,\delta])} + \vert \vert  X^{(1)} - X^{(2)} \vert \vert_{C([0,\delta])} ), 
\end{eqnarray*}
where $\vert\vert X \vert\vert_{C([0,\delta])}$ is the norm of $X$ on the space $C([0,\delta])$. As $t \le \delta$, we conclude that $\vert \vert  \tilde{X}^{(1)} - \tilde{X}^{(2)} \vert\vert_{C([0,\delta])} \le  \delta \, \Lip(W) (\vert \vert  \tilde{X}^{(1)} - \tilde{X}^{(2)} \vert\vert_{C([0,\delta])} + \vert \vert  X^{(1)} - X^{(2)} \vert\vert_{C([0,\delta])})$. As soon as $\delta \, \Lip(W)< 1$, we have $$\vert \vert  \tilde{X}^{(1)} - \tilde{X}^{(2)} \vert\vert_{C([0,\delta])} \le \frac{\delta \, \Lip(W)}{1-\delta \, \Lip(W)} \vert \vert  X^{(1)} - X^{(2)} \vert\vert_{C([0,\delta])}.$$ We choose $\delta$ such that  $\delta \, \Lip(W) < 1/3$ and then $\chi$ is a contraction, as stated, with $\Lip(\chi) \le 1/2$.  
So, we have obtained existence and uniqueness of the solution
on~$[0,\delta]$.
\end{proof}

\section*{Appendix 2: Non-symmetric counter-example}\label{ss:non-sym}

We end this paper with an example showing that for a non-symmetric interaction potential $W$, the conclusion of Theorem~\ref{t:main} does not hold. 

Consider a non-symmetric quadratic interaction potential $W(x) = \frac{1}{2}(x-1)^2$. Then, the averages of the process $(X_t)_t$ defined by~\eqref{eq:AABM1}, $\frac{1}{t} \int_0^t X_s \mathrm{d}s = c_t -1$ tend to $+\infty$. 

To motivate this behaviour, heuristically, we first note that, for any finite-variance measure $\nu$, the convolution $W*\nu$ equals $$W*\nu(x) = \frac{1}{2}(x-1)^2 - (x-1)\mathbb{E}(\nu) + \frac{1}{2}\mathbb{E}(\nu^2) = \frac{x^2}{2} - (\mathbb{E}(\nu) +1) x + \const$$ and hence $\Pi(\nu)$ is the Gaussian law $\mathcal{N}(1+\mathbb{E}(\nu),1)$. Thus, if we consider a trajectory of the approximating flow $\dot{\nu}_t = \frac{1}{t}(\Pi(\nu_t) - \nu_t)$, we have for its mean value $$\frac{\mathrm{d}}{\mathrm{d}t}\mathbb{E}\nu_t = \frac{1}{t}(\mathbb{E} \nu_t + 1 - \mathbb{E}\nu_t) = \frac{1}{t},$$ and so $\mathbb{E}\nu_t \sim \log t$.

For a formal proof, note that (as the interaction potential is a polynomial of degree 2) the evolution of the couple $(X_t,c_t)$, where $c_t = c(\mu_t) = \mathbb{E}\mu_t +1$ is Markovian:
\begin{eqnarray*}
\left\{
\begin{array}{l}
\mathrm{d}X_t = \mathrm{d}B_t - (X_t-c_t) \mathrm{d}t, \\ 
\dot{c}_t = \frac{1}{t}(X_t-c_t+1).
\end{array}
\right.
\end{eqnarray*}
Changing $X_t$ to $Y_t = X_t -c_t$, we obtain 
\begin{eqnarray*}
\left\{
\begin{array}{l}
\mathrm{d}Y_t = \mathrm{d}B_t - \left(Y_t + \frac{1}{t}(Y_t+1)\right) \mathrm{d}t, \\ 
\dot{c}_t = \frac{1}{t}(Y_t+1).
\end{array}
\right.
\end{eqnarray*}
The equation on $Y$ does not contain $c_t$. So, explicit solution of this system and rigorous justification of the desired properties become an exercise.

\bigskip

\textsc{Victor Kleptsyn}: IRMAR (UMR 6625 CNRS), Universit\'e Rennes 1, Campus de Beaulieu, F-35042 Rennes Cedex, France. \\ Email: victor.kleptsyn@univ-rennes1.fr

\textsc{Aline Kurtzmann}:  Institut Elie Cartan, Universit\'e Nancy 1, B.P. 70239, F-54506 Vandoeuvre-l\`es-Nancy Cedex, France. \\ Email: aline.kurtzmann@iecn.u-nancy.fr


\begin{thebibliography}{99}

\bibitem{ane-etc} \textsc{An\'e C., Blach\`ere S., Chafai D.,
Foug\`eres P., Gentil Y., Malrieux F., Roberto C. \textit{\&} Scheffer G.}
(2001), Sur les in\'egalit\'es de Sobolev logarithmiques, \textit{Panoramas et Synth\`eses} \textbf{10}, SMF.

\bibitem{BLR} \textsc{Bena\"im M., Ledoux M. \textit{\&} Raimond O.} (2002), Self-interacting diffusions, \textit{Prob. Th. Rel. Fields} \textbf{122}, 1-41.

\bibitem{BR} \textsc{Bena\"im M. \textit{\&} Raimond O.} (2005), Self-interacting diffusions III: symmetric
interactions, \textit{Ann. Prob.} \textbf{33}(5), 1716-1759.

\bibitem{BGV} \textsc{Bolley F., Guillin A. \textit{\&} Villani C.} (2007), Quantitative concentration
inequalities for empirical measures on non compact spaces, \textit{Prob. Th. Rel. Fields} \textbf{137}(3-4), 541-593.

\bibitem{CMV} \textsc{Carrillo J.A., McCann R.J. \textit{\&} Villani C.} (2003), Kinetic equilibration rates for granular media and related equations: entropy dissipation and mass transportation estimates, \textit{Rev. math. Iberoam.} \textbf{19}(3), 971-1018.

\bibitem{CG} \textsc{Cattiaux P. \textit{\&} Guillin A.} (2008), Deviation bounds for additive functionals of Markov processes, \textit{ESAIM PS} \textbf{12}, 12-29.

\bibitem{CGM} \textsc{Cattiaux P., Guillin A. \textit{\&} Malrieu F.} (2008), Probabilistic approach for granular media equations in the non uniformly convex case , \textit{Prob. Th. Rel. Fields} \textbf{140}(1-2), 19-40.

\bibitem{DR} \textsc{Durrett R.T. \textit{\&} Rogers L.C.G.} (1992), Asymptotic behaviour of Brownian polymers, \textit{Prob. Th. Rel. Fields} \textbf{92}(3), 337-349.

\bibitem{AK} \textsc{Kurtzmann A.} (2009), The ODE method for some self-interacting diffusions on
$\Rn$, \textit{Ann. Inst. Henri Poincar\'e, Prob. Stat.}, to appear.

\bibitem{MC} \textsc{McCann R.} (1997), A convexity principle for interacting gases,
\textit{Adv. Math.} \textbf{128}, 153-179.

\bibitem{meT} \textsc{Meyn S.P. \textit{\&} Tweedie R.L.} (1993), \textit{Markov Chains and Stochastic Stability}, Springer-Verlag.

\bibitem{Pem} \textsc{Pemantle R.} (2007), A survey of random processes with reinforcement, \textit{Prob. Surveys} \textbf{4}, 1-79.

\bibitem{OTwist} \textsc{Raimond O.} (1997), Self-attracting diffusions:
case of constant interaction, \textit{Prob. Th. Rel. Fields} \textbf{107}, 177-196.

\bibitem{RW} \textsc{Rogers L.C.G. \textit{\&} Williams D.} (2000),\textit{ Diffusions, Markov processes and Martingales}, 2nd edition, Vol. 2 ``It\^o Calculus", Cambridge Univ. Press.

\bibitem{TTV} \textsc{Tarr\`es P., T\'oth B. \textit{\&} Valk\'o B.} (2009), Diffusivity bounds for 1d Brownian polymers, preprint.

\bibitem{V} \textsc{Villani C.} (2008), \textit{Optimal Transport, Old and new}, Grundlehren der Math. Wissenschaften , Vol. 338, Springer.

\bibitem{Wu} \textsc{Wu L.} (2000), A deviation inequality for non-reversible Markov process, \textit{Ann. Inst. Henri Poincar\'e, Prob. Stat.} \textbf{36}(4), 435-445.

\bibitem{yosida} \textsc{Yosida K.} (1995), \textit{Functional Analysis}, 6th edition, Springer.
\end{thebibliography}
\end{document}